\documentclass{amsart}
\usepackage{amssymb}
\usepackage{amsfonts}
\usepackage{amssymb}
\usepackage{amsmath}
\usepackage{amsthm}
\usepackage{enumerate}
\usepackage{tabularx}
\usepackage{centernot}
\usepackage{mathtools}
\usepackage{stmaryrd}
\usepackage{amsthm,amssymb}
\usepackage{etoolbox}
\usepackage{tikz}
\usepackage{amssymb}
\usepackage{tablefootnote}
\usetikzlibrary{matrix}
\usepackage{graphics,graphicx}
\usepackage[all]{xy}
\usepackage{tikz-cd}
\usepackage{tikz}
\usepackage{marginnote}
\definecolor{mygray}{gray}{0.85}
\usepackage[linecolor=black,backgroundcolor=mygray,colorinlistoftodos,prependcaption,textsize=small]{todonotes}
\usepackage{xargs}                      

\renewcommand{\leq}{\leqslant}
\renewcommand{\geq}{\geqslant}

\makeatletter
\def\subsection{\@startsection{subsection}{3}%
  \z@{.5\linespacing\@plus.7\linespacing}{.3\linespacing}%
  {\bfseries\centering}}
\makeatother

\makeatletter
\def\subsubsection{\@startsection{subsubsection}{3}%
  \z@{.5\linespacing\@plus.7\linespacing}{.3\linespacing}%
  {\centering}}
\makeatother

\makeatletter
\def\myfnt{\ifx\protect\@typeset@protect\expandafter\footnote\else\expandafter\@gobble\fi}
\makeatother

\newtheorem{theorem}{Theorem}%[section]

\newtheorem{corollary}[theorem]{Corollary}
\newtheorem{definition}[theorem]{Definition}
\newtheorem{lemma}[theorem]{Lemma}
\newtheorem{proposition}[theorem]{Proposition}

\newtheorem{problem}[theorem]{Problem}

\newtheorem{observation}[theorem]{Observation}
\newtheorem{fact}[theorem]{Fact}
\newtheorem{remark}[theorem]{Remark}
\newtheorem{notation}[theorem]{Notation}
\newtheorem{convention}[theorem]{Convention}

\newtheorem{conjecture}[theorem]{Conjecture}

\newtheorem{definition/fact}[theorem]{Definition/Fact}

\newcounter{claimcounter}
\usepackage{eso-pic}
\begin{document}
%%%%%%%%%%%%%%%%%%

\begin{abstract} We give strong necessary conditions on the admissibility of a Polish group topology for an arbitrary graph product of groups $G(\Gamma, G_a)$, and use them to give a characterization modulo a finite set of nodes. As a corollary, we give a complete characterization in case all the factor groups $G_a$ are countable.
\end{abstract}

\title{Polish Topologies for Graph Products of Groups}
\thanks{Partially supported by European Research Council grant 338821. No. 1121 on Shelah's publication list.}

\author{Gianluca Paolini}
\address{Einstein Institute of Mathematics,  The Hebrew University of Jerusalem, Israel}

\author{Saharon Shelah}
\address{Einstein Institute of Mathematics,  The Hebrew University of Jerusalem, Israel \and Department of Mathematics,  Rutgers University, U.S.A.}

%\date{\today}
\maketitle

%\tableofcontents

\section{Introduction}

	\begin{definition}\label{def_cyclic_prod} Let $\Gamma = (V, E)$ be a graph and $\{ G_a: a \in \Gamma \}$ a set of non-trivial groups each presented with its multiplication table presentation and such that for $a \neq b \in \Gamma$ we have $e_{G_a} = e = e_{G_b}$ and $G_a \cap G_b = \{ e \}$. We define the {\em graph product of the groups $\{ G_a: a \in \Gamma \}$ over $\Gamma$}, denoted $G(\Gamma, G_a)$, via the following presentation:
	$$ \text{ generators: } \bigcup_{a \in V} \{ g : g \in G_a \},$$
	$$ \text{ relations: } \bigcup_{a \in V} \{ \text{the relations for } G_a \} \cup \bigcup_{\{a, b \} \in E} \{ gg' = g'g : g \in G_a \text{ and } g' \in G_b \}.$$
\end{definition}

	This paper is the sixth in a series of paper written by the authors which address the following problems:
	
	\begin{problem}\label{problem} Characterize the graph products of groups $G(\Gamma, G_a)$ admitting a Polish group topology (resp. a non-Archimedean Polish group topology).
\end{problem}

	\begin{problem}\label{problem_metric} Determine which graph products of groups $G(\Gamma, G_a)$ are embeddable into a Polish group (resp. into a non-Archimedean Polish group).
\end{problem}

	The beginning of the story is the following question\footnote{The non-Archimedean version of this question was originally formulated by David Evans.}: can a Polish group be an uncountable free group? This was settled in the negative by Shelah in \cite{shelah}, in the case the Polish group was assumed to be non-Archimedean, and in general in \cite{shelah_1}. Later this negative result has been extended by the authors to the class of so-called right-angled Artin groups \cite{paolini&shelah}. After the authors wrote \cite{paolini&shelah}, they discovered that the impossibility results thein follow from an old  important result of Dudley \cite{dudley}. In fact, Dudley's work proves more strongly that any homomorphism from a Polish group $G$ into a right-angled Artin group $H$ is continuous with respect to the discrete topology on $H$.
The setting of \cite{paolini&shelah} has then been further generalized by the authors in \cite{paolini&shelah_cyclic} to the class of graph products of groups $G(\Gamma, G_a)$ in which all the factor groups $G_a$ are cyclic, or, equivalently, cyclic of order a power of prime or infinity. In this case the situation is substantially more complicated, and the solution of the problem establishes that $G = G(\Gamma, G_a)$ admits a Polish group topology if and only if it admits a non-Archimedean Polish group topology if and only if $G = G_1 \oplus G_2$ with $G_1$ a countable graph product of cyclic groups and $G_2$ a direct sum of finitely many continuum sized vector spaces over a finite field. Concerning Problem \ref{problem_metric}, in \cite{paolini&shelah_metric_cyclic} the authors give a complete solution in the case all the $G_a$ are cyclic, proving that $G(\Gamma, G_a)$ is embeddable into a Polish group if and only if it is embeddable into a non-Archimedean Polish group if and only if $\Gamma$ admits a metric which induces a separable topology in which $E_{\Gamma}$ is closed. We hope to conclude this series of studies with an answer to Problem \ref{problem_metric} at the same level of generality of this paper.
	The logical structure of the references just mentioned (plus the present paper) is illustrated in Figure \ref{log_str}, where we use the numbering of Shelah's publication list, and one-direction arrows mean generalization and two-direction arrows mean solutions to Problem \ref{problem}/Problem \ref{problem_metric} at the same level of generality.

\begin{figure}[ht]
	\begin{center}
	
\begin{displaymath}
    \xymatrix{& [1121] & \\
              & [1115] \ar@{=>}[u] \ar@{<=>}[r]  & [1117] \\
              & [1112] \ar@{=>}[u] \\ % \ar@{<=}[dr] \ar@{<=}[dl] \\
              & [771] \ar@{=>}[u] \\
              & [744] \ar@{=>}[u]}
\end{displaymath}

\end{center}  \caption{Logical structure of the references.}\label{log_str}
\end{figure}
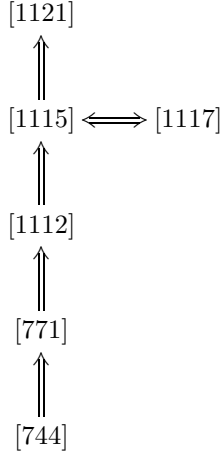

	In the present study we focus on Problem \ref{problem}, proving the following theorems:

	\begin{notation}\label{def_bounded} \begin{enumerate}[(1)]
	\item We denote by $\mathbb{Q} = G^*_{\infty}$ the rational numbers, by $\mathbb{Z}^{\infty}_p = G^*_p$ the divisible abelian $p$-group of rank $1$ (for $p$ a prime), and by $\mathbb{Z}_{p^k} = G^*_{(p, k)}$ the finite cyclic group of order $p^k$ (for $p$ a prime and $k \geq 1)$.
	\item We let $S_* = \{ (p, k): p \text{ prime and } k \geq 1 \} \cup \{ \infty \}$ and $S_{**} = S_{*} \cup \{ p : p \text{ prime}\}$;
	\item For $s \in S_{**}$ and $\lambda$ a cardinal, we let $G^*_{s, \lambda}$ be the direct sum of $\lambda$ copies of $G^*_s$.
	\end{enumerate}
\end{notation}

	\begin{theorem} \label{th_first_venue} Let $G = G(\Gamma, G_a)$ and suppose that $G$ admits a Polish group topology. Then for some countable $A \subseteq \Gamma$ and $1 \leq n < \omega$ we have:
	\begin{enumerate}[(a)]
	\item for every $a \in \Gamma$ and $a \neq b \in \Gamma - A$, $a$ is adjacent to $b$;
	\item if $a \in \Gamma - A$, then $G_a = \bigoplus \{ G^*_{s, \lambda_{a, s}} : s \in S_* \}$ (cf. Notation \ref{def_bounded});
	\item if $\lambda_{a, (p, k)} > 0$, then $p^k \mid n$;
	\item if in addition $A = \emptyset$, then for every $s \in S_*$ we have that $\sum\{\lambda_{a, s} : a \in \Gamma \}$ is either $\leq \aleph_0$ or $2^{\aleph_0}$.
\end{enumerate}
\end{theorem}

%	For $G = G(\Gamma, G_a)$ and $A \subseteq \Gamma$, we let $G_{\bar{A}} = G(\Gamma_{\Gamma-A}, G_a)$, i.e. the graph product with respect to the induced subgraph of $\Gamma$ on vertex set ${\Gamma-A}$.

	The following more involved theorems give more information on the possible graph products decompositions of a group $G$ admitting a Polish group topology, and it can be seen as a solution modulo a finite set of nodes to Problem \ref{problem}.

\begin{theorem}\label{th_second_venue} \begin{enumerate}[(1)]
	\item Let $G = G(\Gamma, G_a)$. If $G$ admits a Polish group topology, then there is $\bar{A} = (A_0, A_5, A_6, A_7, A_8, A_9)$ such that:
	\begin{enumerate}[(a)]
	\item\label{0} $\bar{A}$ is a partition of $\Gamma$;
	\item\label{1} for every $a \in \Gamma$ and $a \neq b \in \Gamma - A_0$, $a$ is adjacent to $b$;
	\item\label{b} $A_5$ and $A_6$ are finite;
	\item\label{c} $A_0$, $A_7$ and $A_8$ are countable, 
	\item\label{d} for each $a \in A_0$, $G_a$ is countable;
	\item\label{e} if $a \in A_7 \cup A_8$, then $G_a = H_a \oplus \bigoplus \{ G^*_{s, \lambda_{a, s}} : s \in S_{**} \}$, for some countable $H_a \leq G_a$;
	\item\label{f} if $a \in A_9$, then $G_a = \bigoplus \{ G^*_{s, \lambda_{a, s}} : s \in S_{**} \}$;
	\item\label{g} for each $s \in S_{**} - S_{*}$, $\sum\{\lambda_{a, s} : a \in A_7 \cup A_8 \cup A_9 \} \leq \aleph_0$;
	\item\label{leq_continuum} for each $s \in S_*$, $\sum\{\lambda_{a, s} : a \in A_7 \cup A_8 \cup A_9 \}$ is $\leq 2^{\aleph_0}$;
	\item\label{n_bound} for some $1 \leq n < \omega$ we have $\sum\{\lambda_{a, (p, k)} : a \in A_7 \cup A_8 \cup A_9 \} > \aleph_0 \Rightarrow p^k \mid n$;
	\item\label{i} we can {\em define explicitly} the $A_i$'s from $\{ G_a : a \in \Gamma \}$.
\end{enumerate}
\item Furthermore, if we assume CH and we let $\bar{A} = (A_0, A_5, A_6, A_7, A_8, A_9)$ be as above and $A = A_0 \cup A_7 \cup A_8 \cup A_9$, then $G(\Gamma \restriction A, G_a)$ admits a non-Archimedean Polish group topology.
\end{enumerate}
\end{theorem}

	\begin{theorem}\label{new_th_second_venue} \begin{enumerate}[(1)]
	\item For given $G = G(\Gamma, G_a)$ the following conditions are equivalent:
	\begin{enumerate}[(a)]
	\item for some finite $B_1 \subseteq \Gamma$, for every finite $B_2$ such that $B_1 \subseteq B_2 \subseteq \Gamma$, $G(\Gamma \restriction \Gamma - B_2, G_a)$ admits a Polish group topology;
	\item there is $\bar{A}$ as in Theorem \ref{th_second_venue} and for some finite $B \supseteq A_5 \cup A_6$, for every $s \in S_*$ the cardinal $\lambda_s^B = \sum \{ \lambda_{a, s} : a \in (A_7 \cup A_8 \cup A_9) - B \}$ \mbox{is either $\aleph_0$ or $2^{\aleph_0}$.}
%	\item for some finite $B_1 \subseteq \Gamma$, for every finite $B_2$ such that $B_1 \subseteq B_2 \subseteq \Gamma$, there is finite $B_3$ such that $B_2 \subseteq B_3 \subseteq \Gamma$ and $G(\Gamma \restriction \Gamma - B_3, G_a)$ admits a Polish group topology.
\end{enumerate}
	\item If $B_0 \subseteq \Gamma$ is finite, $\bar{A}$ is as in Theorem \ref{th_second_venue} for $G(\Gamma \restriction \Gamma - B_0, G_a)$ and we let $B_1 = B_0 \cup A_5 \cup A_6$ (which is a finite subset of $\Gamma$), then the following conditions on $B \subseteq \Gamma - B_1$ are equivalent:
	\begin{enumerate}[(a)]
	\item $G(\Gamma \restriction B)$ admits a Polish group topology;
	\item for every $s \in S_*$ the cardinal $\lambda_s^B = \sum \{ \lambda_{a, s} : a \in B \}$
is either $\aleph_0$ or $2^{\aleph_0}$.
%\todo[inline]{I do not understand. If you were to write: $\bar{A}$ is as in Theorem \ref{th_second_venue} for $G(\Gamma \restriction \Gamma, G_a)$ and we let $B_1 = B_0 \cup A_5 \cup A_6$..., then I would agree. But as written is confusing, since $A_5$ and $A_6$ we get after we choose a $G(\Gamma \restriction \Gamma, G_a)$.}
\end{enumerate}
\end{enumerate}	
\end{theorem}

	\begin{remark}\label{no_char} Let:
	\begin{enumerate}[(a)]
	\item $s \in S_*$;
	\item $\aleph_0 < \lambda < 2^{\aleph_0}$;
	\item $\Gamma$ a complete graph on $\omega_1$;
	\item $G_0 = G_{s, 2^{\aleph_0}} \oplus G_{*}$;
	\item $G_*$ an uncountable centerless group admitting a Polish group topology;
	\item $G_{\alpha} = G_{s, \lambda}$, for $\alpha \in [1, \omega_1)$.
\end{enumerate} 
Then $G(\Gamma, G_a)$ admits a Polish group topology, but letting 
$\bar{A}$ be the partition from Theorem \ref{th_second_venue} we have that $\sum\{\lambda_{a, s} : a \in A_7 \cup A_8 \cup A_9 \} = \lambda < 2^{\aleph_0}$, and so for $A = A_0 \cup A_7 \cup A_8 \cup A_9$, we have that $G(\Gamma \restriction A, G_a)$ does {\em not} admit a Polish group topology (in fact in this case $A_0 = A_6 = A_7 = A_8 = \emptyset$, $A_5 = \{ 0 \}$ and $A_9 = [1, \omega_1)$, cf. the explicit definition of the $A_i$'s in the proof of Theorem \ref{th_second_venue}).
\end{remark}

%\begin{definition}
%\begin{enumerate}[(1)]
%	\item We say that $(G, \mathfrak{d})$ is d.s. Polish (direct summand of Polish) if, for some Polish group $(G_1, \mathfrak{d}_1)$, $G$ is a direct summand of $G_1$ and $\mathfrak{d} = \mathfrak{d}_1 \restriction G$.
%	\item We say that $(G, \mathfrak{d})$ is non-Archimedean d.s. Polish when $(G, \mathfrak{d})$ is as in (1) and in addition $(G_1, \mathfrak{d}_1)$ is non-Archimedean.
%	\item We say that the group $G$ admits a d.s. Polish topology (resp. a non-Archimedean d.s. Polish topology) if there exists a metric $\mathfrak{d}$ on $G$ such that $(G, \mathfrak{d})$ is d.s. Polish (resp. non-Archimedean d.s. Polish).
%\end{enumerate}
%\end{definition}

%	\begin{theorem}\label{th_dir_summand}
%	 If $(G, \mathfrak{d})$ is d.s. Polish and $G = G(\Gamma, G_a)$, then:
%	\begin{enumerate}[(1)]
%	\item conclusions (a), (b) and (c) of Theorem \ref{th_first_venue} hold;
%	\item with the possible exception of clause (h), all the conclusions of Theorem \ref{th_second_venue} hold, where in the ``furthermore part'' we replace ``non-Archimedean Polish group topology'' by ``non-Archimedean d.s. Polish topology''.
%	\end{enumerate}
%\end{theorem}

	From our theorems and their proofs we get the following corollaries.
	
	\begin{corollary}\label{cor_third_venue} Let $G = G(\Gamma, G_a)$ with all the $G_a$ countable. Then $G$ admits a Polish group topology if and only if $G$ admits a non-Archimedean Polish group topology if and only if there exist a countable $A \subseteq \Gamma$ and $1 \leq n < \omega$ such that:
	\begin{enumerate}[(a)]
	\item for every $a \in \Gamma$ and $a \neq b \in \Gamma - A$, $a$ is adjacent to $b$;
	\item if $a \in \Gamma - A$, then $G_a = \bigoplus \{ G^*_{s, \lambda_{a, s}} : s \in S_* \}$;
	\item if $\lambda_{a, (p, k)} > 0$, then $p^k \mid n$;
	\item for every $s \in S_*$, $\sum\{\lambda_{a, s} : a \in \Gamma - A \}$ is either $\leq \aleph_0$ or $2^{\aleph_0}$.
\end{enumerate}
\end{corollary}

	\begin{corollary}\label{cor_bounded_divis} Let $G$ be an abelian group which is a direct sum of countable groups, then $G$ admits a Polish group topology if only if $G$ admits a non-Archimedean Polish group topology if and only if there exists a countable $H \leq G$ and $1 \leq n < \omega$ such that:
	$$G = H \oplus \bigoplus_{\alpha < \lambda_{\infty}} \mathbb{Q} \oplus \bigoplus_{p^k \mid n} \bigoplus_{\alpha < \lambda_{(p, k)}} \mathbb{Z}_{p^k},$$
with $\lambda_{\infty}$ and $\lambda_{(p, k)} \leq \aleph_0$ or $2^{\aleph_0}$.
\end{corollary}

	\begin{corollary}[Slutsky \cite{slutsky}]\label{corollary_free_prod} If $G$ is an uncountable group admitting a Polish group topology, then $G$ can not be expressed as a non-trivial free product.
\end{corollary}

%	\begin{remark}\label{remark_hi_Polish} Notice that Theorem \ref{th_dir_summand} cannot be improved to ``there exists a surjective homomorphism from a Polish group $(G_1, \mathfrak{d}_1)$ onto $G$'', since $G^*_{p, 2^{\aleph_0}}$ is a homomorphic image of $G^*_{\infty, 2^{\aleph_0}}$, and the latter admits a Polish group topology.
%\end{remark}

	 The following problem gets in the way of a complete characterization of the groups $G = G(\Gamma, G_a)$ admitting a Polish group topology in the case no further assumptions are made on the factors $G_a$. We have:
	
	\begin{fact}\label{counter_ex} Let $s_1 \neq s_2 \in S_{*}$ and $\lambda$ a cardinal (cf. Notation \ref{def_bounded}).
\begin{enumerate}[(1)]
\item If $\aleph_0 < \lambda < 2^{\aleph_0}$, then $G_{s_1, \lambda} \oplus G_{s_1, 2^{\aleph_0}} \cong G_{s_1, 2^{\aleph_0}}$ admits a Polish group topology, but $G_{s_1, \lambda}$ does not admit one such topology.
\item If $\aleph_0 < \lambda < 2^{\aleph_0}$, $H_1 = G_{s_1, 2^{\aleph_0}} \oplus G_{s_2, \lambda}$ and $H_2 = G_{s_1, \lambda} \oplus G_{s_2, 2^{\aleph_0}}$, then $H_1 \oplus H_2$ admits a Polish group topology, but neither $H_1$ nor $H_2$ admit one such topology.
\end{enumerate}
\end{fact}
	
	Hence, a general characterization seems to depend on the failure of CH. Despite this, our impression is that CH would not help. This leads to a series of conjectures on the possible direct summands of a Polish group $G$:
	
	\begin{conjecture}[Polish Direct Summand Conjecture]\label{conj_summands} Let $G$ be a group admitting a Polish group topology.
	\begin{enumerate}[(1)]
	\item If $G$ has a direct summand isomorphic to $G^*_{s, \lambda}$, for some $\aleph_0 < \lambda \leq 2^{\aleph_0}$ and $s \in S_*$, then it has one of cardinality $2^{\aleph_0}$.
	\item If $G = G_1 \oplus G_2$ and $G_2 = \bigoplus \{ G^*_{s, \lambda_{s}} : s \in S_* \}$, then for some $G'_1, G'_2$ we have:
	\begin{enumerate}[(i)]
	\item $G_1 = G'_1 \oplus G'_2$;
	\item $G'_1$ admits a Polish group topology;
	\item $G'_2 = \bigoplus \{ G^*_{s, \lambda'_{s}} : s \in S_* \}$.
	\end{enumerate}
	\item If $G = G_1 \oplus G_2$, then for some $G'_1, G'_2$ we have:
	\begin{enumerate}[(i)]
	\item $G_1 = G'_1 \oplus G'_2$;
	\item $G'_1$ admits a Polish group topology;
	\item $G'_2 = \bigoplus \{ G^*_{s, \lambda_{s}} : s \in S_* \}$.
	\end{enumerate}
	\end{enumerate}
\end{conjecture}

	The paper is organized as follows. In Section \ref{preliminaries} we prove some preliminaries results to be used in later sections. In Section \ref{first_venue} we prove Theorem \ref{th_first_venue}. In Section \ref{second_venue} we prove Theorems \ref{th_second_venue} and \ref{new_th_second_venue}. In Section \ref{third_venue} we prove Corollaries \ref{cor_third_venue}, \ref{cor_bounded_divis} and \ref{corollary_free_prod}.
	
		In a work in preparation we deal with Conjecture \ref{conj_summands}, and mimic Theorems \ref{th_first_venue} and \ref{th_second_venue} in a weaker context, i.e. the topology on $G$ need not be Polish.

\section{Preliminaries}\label{preliminaries}

	In notation and basic results we follow \cite{bark}. Given $A \subseteq \Gamma$ we denote the induced subgraph of $\Gamma$ on vertex set $A$ as $\Gamma \restriction A$.
	
	\begin{fact}\label{retract_fact} Let $G = G(\Gamma, G_a)$, $A \subseteq \Gamma$ and $G_A = (\Gamma \restriction A, G_a)$. Then there exists a unique homomorphism $\mathbf{p} = \mathbf{p}_A: G \rightarrow G_A$ such that $\mathbf{p}(g) = g$ if $g \in G_A$, and $\mathbf{p}(g) = e$ if $g \in G_{\Gamma - A}$.
\end{fact}

	\begin{proof} For arbitrary $G = G(\Gamma, G_a)$, let $\Omega_{(\Gamma, G_a)}$ be the set of equations from Definition \ref{def_cyclic_prod} defining $G(\Gamma, G_a)$. Then for the $\Omega_{(\Gamma, G_a)}$ of the statement of the fact we have $\Omega_{(\Gamma, G_a)} = \Omega_1 \cup \Omega_2 \cup \Omega_3$, where:
	\begin{enumerate}[(a)]
	\item $\Omega_1 = \Omega_{(\Gamma \restriction A, G_a)}$;
	\item $\Omega_2 = \Omega_{(\Gamma \restriction {\Gamma - A}, G_a)}$;
	\item $\Omega_3 = \{ bc = cb : b E_{\Gamma} c \text{ and } \{ b, c \} \not\subseteq A \}$.
\end{enumerate}
Notice now that $\mathbf{p}$ maps each equation in $\Omega_1$ to itself and each equation in $\Omega_2 \cup \Omega_3$ to a trivial equation, and so $p$ is an homomorphism (clearly unique).
\end{proof}

	\begin{definition} A word in $G(\Gamma, G_a)$ is either $e$ (the empty word) or a formal product $g_1 \cdots g_n$ with each $g_i \in G_{a_{i}}$ for some $a_{i} \in \Gamma$. The elements $g_i$ are called the syllables of the word. The length of the word $g_1 \cdots g_n$ is $|g_1 \cdots g_n| = n$, with the length of the empty word defined to be $0$. If $g \in G(\Gamma, G_a)$ satisfies $G(\Gamma, G_a) \models g = g_1 \cdots g_n$, then we say that the word $g_1 \cdots g_n$ represents (or spells) $g$. We will abuse notation and do not distinguish between a word and the element of $G$ that it represents.
\end{definition}

\begin{definition}\label{moves} The word $g_1 \cdots g_n$ is a normal form if it cannot be changed into a shorter word by applying a sequence of moves of the following type:
\begin{enumerate}[$(M_1)$]
\item delete the syllable $g_i = e$.
\item if $g_i, g_{i+1} \in G_a$, replace the two syllables $g_i$ and $g_{i+1}$ by the single syllable $g_i g_{i+1} \in G_a$.
\item if $g_i \in G_a$, $g_{i+1} \in G_b$ and $aE_{\Gamma}b$, exchange $g_i$ and $g_{i+1}$.
\end{enumerate}
\end{definition}
	
	\begin{fact}[Green \cite{green} for (1) and Hermiller and Meier \cite{hermiller} for (2)]
\begin{enumerate}[(1)]
	\item If a word in $G(\Gamma, G_a)$ is a normal form and it represents the identity element, then it is the empty word. 
	\item If $w_1$ and $w_2$ are two words representing the same element $g \in G(\Gamma, G_a)$, then $w_1$ and $w_2$ can be reduced to identical normal forms using moves $(M_1)-(M_3)$.
\end{enumerate}
\end{fact}

	\begin{definition} Let $g \in G(\Gamma, G_a)$. We define:
	\begin{enumerate}[(1)]
	\item $sp(g) = \{ a \in \Gamma : g_i \text{ is a syllable of a normal form for } g \text{ and } g_i \in G_a - \{ e \} \}$;
	\item $lg(g) = |w|$, for $w$ a normal form for $g$;
	\item $F(g) = \{ g_1 : g_1 \cdots g_n \text{ is a normal form for } g\}$;
	\item $L(g) = \{ g_n : g_1 \cdots g_n \text{ is a normal form for } g\}$;
	\item $\hat{L}(g) = \{ g_n^{-1}:  g_n \in L(g) \}$.
\end{enumerate}
(Here $F$ and $L$ stand for ``first'' and ``last'', respectively.)
\end{definition}

\begin{definition}\label{def_cyclically}
\begin{enumerate}[(1)]
\item We say that the word $w$ is {\em weakly cyclically reduced} when:
$$F(w) \cap \hat{L}(w) = \emptyset.$$
 \item We say that the word $g_1 \cdots g_n$ is {\em cyclically reduced} if no combination of moves $(M_1)-(M_4)$ results in a shorter word, where $(M_1)-(M_3)$ are as in Definition~\ref{moves} and the move $(M_4)$ is as follows:
\begin{enumerate}[$(M_4)$]
\item replace $g_1 \cdots g_n$ by either $g_2 \cdots g_ng_1$ or $g_ng_1 \cdots g_{n-1}$.
\end{enumerate}
\item We say that $g \in G(\Gamma, G_a)$ is $(a, b)$-cyclically reduced (or $(G_a, G_b)$-cyclically reduced) when $g \neq e$, $F(g) \subseteq G_a - \{ e \}$ and $L(g) \subseteq G_b - \{ e \}$.
\end{enumerate}
\end{definition}

	\begin{observation}\label{obs_normal_forms} Notice that if $g \in G(\Gamma, G_a)$ is spelled by a cyclically reduced (resp. a weakly cyclically reduced) normal form, then any of the normal forms spelling $g$ is cyclically reduced (resp. weakly cyclically reduced).
\end{observation}

	\begin{definition}\label{def_cyc_el} Recalling Observation~\ref{obs_normal_forms}, we say that $g \in G(\Gamma, G_a)$ is cyclically reduced (resp. weakly cyclically reduced) if any of the normal forms spelling $g$ is cyclically reduced (resp. weakly cyclically reduced).
\end{definition}

	\begin{remark} 
	\begin{enumerate}[(1)]
	\item Notice that if $g$ is cyclically reduced, then $g$ cannot be written as a normal form $h_1h_2 \cdots h_{n-1}h_n$ with $h_1, h_n \in G_a$ for some $a \in \Gamma$, since otherwise $lg(h_2 \cdots h_{n-1}h_nh_1) < lg(h_1h_2 \cdots h_{n-1}h_n)$.
	\item Notice that if $g$ is weakly cyclically reduced, then $g$ cannot be written as a normal form $h_1h_2 \cdots h_{n-1}h^{-1}_1$, since otherwise $F(w) \cap \hat{L}(w) \neq \emptyset$.
	\item Hence, if $g$ is cyclically reduced and spelled by the normal form $h_1h_2 \cdots h_{n-1}h_n$, then, unless $n = 1$ and $h_1 = h^{-1}_1$, we have that $g$ is weakly cyclically reduced.
\end{enumerate}
\end{remark}

	\begin{proposition}\label{prop_(ab)_reduced} Let $a \neq b \in \Gamma$, $\{a, b\} \not \in E_{\Gamma}$ and $g_1 u g_2 \in G(\Gamma, G_a)$. 
	Assume that $F(g_1), F(u), F(g_2) \subseteq G_a - \{ e \}$, $L(g_1), L(u), L(g_2) \subseteq G_b - \{ e \}$ and $p \geq 2$, then:
	\begin{enumerate}[(a)]
	\item $g_1u^pg_2$ is $(a, b)$-cyclically reduced;
	\item if $g_1, u, g_2$ are written as normal forms, then $g_1\underbrace{u \cdots u}_pg_2$ is a normal form;
	\item $lg(g_1u^pg_2) = lg(g_1) + plg(u) + lg(g_2) > lg(g_1ug_2) > lg(u)$.
	\end{enumerate}
\end{proposition}

	\begin{proof} Clear.
\end{proof}

\begin{convention}\label{convention_words} Given a sequence of words $w_1, ..., w_k$ with some of them possibly empty, we say that the word $w_1 \cdots w_k$ is a normal form (resp. a (weakly) cyclically reduced normal form) if after deleting the empty words the resulting word is a normal form (resp. a (weakly) cyclically reduced normal form).
\end{convention}
	
	\begin{fact}[\cite{bark}{[Corollary 24]}]\label{bark_fact} Any element $g \in G(\Gamma, G_a)$ can be written in the form $w_1 w_2 w_3 w'_2 w^{-1}_1$, where:
	\begin{enumerate}[(1)]
	\item $w_1 w_2 w_3 w'_2 w^{-1}_1$ is a normal form;
	\item the element $w_3 w'_2 w_2$ is cyclically \mbox{reduced (cf. Observation~\ref{obs_normal_forms} and Def.~\ref{def_cyc_el});}
	\item $sp(w_2) = sp(w'_2)$;
	\item if $w_2 \neq e$, then $\Gamma \restriction {sp(w_2)}$ is a complete graph;
	\item $F(w_2) \cap \hat{L}(w'_2) = \emptyset$.
	\end{enumerate}
(Notice that by (5) if $w_2 \neq e$ then $w_2 w_3 w'_2$ is weakly cyclically reduced).
\end{fact}

%	\begin{fact}[\cite{green}{[Lemma 3.16]}]\label{green_fact} Any element $g \in G(\Gamma, G_a)$ can be written as $g_1g_0g_1^{-1}$, for $g_0, g_1 \in G(\Gamma, G_a)$ and $g_0$ cyclically reduced.
%\end{fact}

	\begin{definition} Let $g \in G(\Gamma, G_a)$ and $g = w_1 w_2 w_3 w'_2 w^{-1}_1$ as in Fact \ref{bark_fact}. We let:
\begin{enumerate}[(1)]
\item $csp(g) = sp(w_2 w_3 w'_2)$;
\item $clg(g) = lg(w_2 w_3 w'_2)$.
\end{enumerate}
	(Inspection of the proof of Fact \ref{bark_fact} from \cite{bark} shows that this is well-defined).
\end{definition}

%	We denote the order of a group element $g$ by $o(g)$.

	\begin{proposition}\label{prop_1} Let $G = G(\Gamma, G_a)$, with $\Gamma = \{ a_1, a_2, b_1, b_2 \}$ and $|\{ a_1, a_2, b_1, b_2 \}| = 4$. Suppose also that, for $i = 1, 2$, we have that $a_i$ and $b_i$ are not adjacent. Then:
	\begin{enumerate}[(1)]
	\item If $g \in G$ has finite order, then $csp(g)$ is a complete graph (and so $|csp(g)| \leq 2$).
	\item Let $q < p$ be primes, $g_i \in G_{a_i} - \{ e \}$ and $h_i \in G_{b_i} - \{ e \}$ ($i =1, 2$), and $g = (g_1g_2h_1h_2)^p$. Then for every $d \in G$ such that $csp(g)$ is a complete graph (and so $|csp(g)| \leq 2$) we have that $dg \in G$ does not have a $q$-th root.
	\end{enumerate}
\end{proposition}

	\begin{proof} This can be proved using the canonical representation of $d \in G$ that we get from Fact~\ref{bark_fact}, and analyzing the possible cancellations occurring in the word $dg$, in the style e.g. of the proof of Proposition~\ref{pre_prop_2}. The details are omitted.
\end{proof}

%	\begin{proposition}\label{old_prop_2} Let $G = G(\Gamma, G_a)$, $\Gamma = \{ a, b_1, b_2 \}$, and suppose that $\Gamma$ is either discrete or it has only the edge $b_1Eb_2$. Let then $g_* \in G_a$, $h_i \in G_{b_i} - \{ e \}$ ($i =1, 2$) and $q = g^{-1}_*h_{1}^{-1} h_{2} g_*$. Then for every $k \in G$ of finite order we have that $kq$ does not have a square root.
%\end{proposition}

	\begin{notation} We denote the free product of two group $H_1$ and $H_2$ as $H_1 \ast H_2$. Notice that $H_1 \ast H_2$ is $G(\Gamma, G_a)$ for $\Gamma$ a discrete graph (i.e. no edges) on two vertices $a$ and $b$, and $G_a = H_1$ and $G_b = H_2$. Thus, when we use $lg(g)$, $sp(g)$, etc., for $g \in H_1 \ast H_2$, we mean with respect to the corresponding $G(\Gamma, G_a)$.
\end{notation}

\begin{proposition}\label{pre_prop_2} Let $k_* \geq 2$ be even and $p >> k_*$ (e.g., as an overkill, we might let $p = 36k_* + 100$). Then (A) implies (B), where:
\begin{enumerate}[(A)]
	\item 
	\begin{enumerate}[(a)]
	\item $H = H_1 \ast H_2$;
	\item $g_* \in H_1 - \{ e \}$;
	\item $h_{(\ell, i)} \in H_2 - \{ e \}$, for $\ell < k_*$ and $i = 1, 2$;
	\item $(h_{(\ell, i)} : \ell < k_* \text{ and } i = 1, 2)$ is with no repetitions;
	\item $(h_{(0, 2)})^{-1} \neq h_{(k_*-1, 1)} \neq (h_{(1, 2)})^{-1}$;
	\item for $i \in \{ 1, 2 \}$ and $0 < \ell < k_*$ we have $h_{(0, 2)} \neq (h_{(\ell, 2)})^{-1}$;
	\item for $i \in \{ 1, 2 \}$ and $0 \leq \ell < k_* - 1$ we have $h_{(k_* - 1, 1)} \neq (h_{(\ell, 1)})^{-1}$;
	%\item\label{inverse_assumption} $h_{(\ell_1, i_1)} = (h_{(\ell_2, i_2)})^{-1}$ implies $(\ell_1, i_1) = (\ell_2, i_2)$, for $i_1, i_2 \in \{ 1, 2 \}$ and $\ell_1, \ell_2 \in \{ 0, ..., k_*-1 \}$;
	\item $g_i = h_{(0, i)} g_* h_{(1, i)} g^{-1}_* \cdots h_{(k_* - 2, i)} g_* h_{(k_* - 1, i)} g^{-1}_*$, for $i = 1, 2$;
	\end{enumerate}
\end{enumerate}
\begin{enumerate}[(B)] 
	\item for every $u \in H$, at least one of the following holds:
	\begin{enumerate}[(a)]
	\item $lg(g_1 u^p g_2) > lg(u)$;
	\item $clg(u) \leq 1$, $lg(g_1 u^p g_2) \geq 2k_*$ and $g_1u^pg_2$ is $(H_2, H_1)$-cyclically reduced;
	\item $clg(u) \leq 1$, $lg(g_1 u^p g_2) \geq 2k_*$ and $lg(g_1 u^p g_2) = lg(u)$.
	\end{enumerate}
\end{enumerate}
\end{proposition}

	\begin{proof} Let $u \in H$, write $u = w_1 w_2 w_3 w'_2 w^{-1}_1$ as in Fact \ref{bark_fact} and set $w_2 w_3 w'_2 = w_0$. Clearly the element $g_1 u^p g_2$ is spelled by the following word (thinking of $g_i$ as a word (cf. its definition)):
	$$w_* = g_1w_1\underbrace{w_0 \cdots w_0}_{p}w^{-1}_1 g_2.$$
\newline \underline{\em Case 1}. $lg(w_0) \geq 2$ and $lg(w_0)$ is even.
\newline Notice that in this case the word:
$$w_1\underbrace{w_0 \cdots w_0}_{p}w^{-1}_1$$
is a normal form for $u^p$, and so the only places where cancellations (i.e. consecutive applications of moves $(M_1)$ and $(M_2)$ as in Definition \ref{moves})  may occur in $w_*$ are at the junction of $g_1$ and $w_1$ and at the junction of $w^{-1}_1$ and $g_2$. Since by assumption $lg(g_i) = 2k_*$ ($i = 1, 2$) and $p >> 4k_*$, we get that $lg(g_1 u^p g_2) > lg(u)$. Thus, clause $B(a)$ is true.
\newline \underline{\em Case 2}. $lg(w_0) \geq 3$ and $lg(w_0)$ is odd.
\newline In this case, for some $\ell \in \{ 1, 2 \}$, $F(w_0),L(w_0) \in H_{3-\ell} - \{e\}$, $w_2, w'_2 \in H_\ell - \{ e \}$ and $w'_2 w_2 \neq e$, and so, letting $w'_0$ stand for a normal form for $w_3 w'_2 w_2$ (i.e. $w'_0 = w_3 (w'_2 w_2)$), we have that $lg(w'_0) \geq 2$. Thus, the word:
$$w_1w_2\underbrace{w'_0 \cdots w'_0}_{p-1}w_3w'_2w^{-1}_1$$
is a normal form for $u^p$. Hence, arguing as in Case 1, we see that $lg(g_1 u^p g_2) > lg(u)$. Thus, clause $B(a)$ is true.
\newline \underline{\em Case 3}. $lg(w_0) = 1$, $(w_0)^p \neq e$, and $w_1 = e = w_1^{-1}$.
\newline This case is clear by assumption (A)(e). Clearly in this case clause $B(a)$ is true.
\newline \underline{\em Case 4}. $lg(w_0) = 1$, $(w_0)^p \neq e$ and $w_1 \neq e \neq w_1^{-1}$.
\newline If this is the case, then $(w_0)^p = (w_3)^p = g^p$ for some $g \in H_1 \cup H_2$ (and $w_2 = e = w'_2$). Notice crucially that in $w_* = g_1w_1(w_0)^pw_1^{-1}g_2$ if a cancellation occurs at the junction of $g_1$ and $w_1$ (resp. of $w^{-1}_1$ and $g_2$), then it cannot occur at the junction of $w^{-1}_1$ and $g_2$ (resp. of $g_1$ and $w_1$), since for $i =1, 2$ we have $F(g_i) \subseteq H_2$ and $L(g_i) \subseteq H_1$, whereas $e \neq F(w_1) = \hat{L}(w_1^{-1}) \neq e$. 
\newline \noindent \underline{\em Case 4.1}. No cancellation occurs at the junction of $g_1$ and $w_1$.
\newline Let $m_*$ be the number of cancellations occurring at the junction of $w^{-1}_1$ and $g_2$.
\newline \underline{\em Case 4.1.1}. $2lg(w_1) + 1 > 2k_*$.
\newline Clearly $m_* \leq 2k_*$ and so we have:
\[ \begin{array}{rcl}
	   lg(g_1u^pg_2) & \geq & 2k_* + 2lg(w_1) + 1 - m_* \\
					 & \geq & 2lg(w_1) + 1 \\
					 & =    & lg(u),
\end{array}	\]
and so either clause (B)(a) or B(c) is true.
\newline \underline{\em Case 4.1.2}. $2lg(w_1) + 1 \leq 2k_*$.
\newline First of all, necessarily $2lg(w_1) + 1 < 2k_*$. Furthermore, notice crucially that $m_* < 2lg(w_1) + 1$, because otherwise we would have:
$$h_{0, 2} = \hat{L}(w_1^{-1}) \text{ and } h_{lg(w_1), 2} = (F(w_1))^{-1},$$
contradicting assumption (A)(f). 
\noindent Hence, $g_1$ is an initial segment of a normal form spelling $g_1u^pg_2$ and so we have:
\[ \begin{array}{rcl}
	   lg(u) & =& 2lg(w_1) + 1 \\
					& < & 2k_* \\
					& \leq & lg(g_1u^pg_2),
\end{array}	\]
and so clause $B(a)$ is true.
\newline \underline{\em Case 4.2}. No cancellation occurs at the junction of $g_2$ and $w^{-1}_1$. 
\newline Let $m_*$ be the number of cancellations occurring at the junction of $w^{-1}_1$ and $g_2$.
\newline \underline{\em Case 4.2.1}. $2lg(w_1) + 1 > 2k_*$.
\newline As in Case 4.1.1. 
\newline \underline{\em Case 4.2.2}. $2lg(w_1) + 1 \leq 2k_*$.
\newline Similar to Case 4.1.2, using assumption (A)(g).
\newline \underline{\em Case 5}. $lg(w_0) = 0$, or $lg(w_0) = 1$ and $(w_0)^p = e$.
\newline If either of these cases happen, then $w_* = g_1 g_2$ is a normal form of length $4k_*$, and so clearly $lg(g_1 u^p g_2) \geq 2k_*$ and $g_1u^pg_2$ is $(H_2, H_1)$-cyclically reduced. Thus, clause $B(b)$ is true.
\end{proof}

\begin{proposition}\label{prop_2} The set of equations $\Omega$ has no solution in $H$, when:
	\begin{enumerate}[(a)]
	\item $k(n) \geq 2$ is even and $p(n) >> k(n)$, for $n < \omega$;
	\item $n < m < \omega$ implies $k(n) < k(m)$;
	\item $H = H_1 \ast H_2$;
	\item for every $n < \omega$ we have:
	\begin{enumerate}[(d.1)] 
	\item $g_{(n, *)} \in H_1 - \{ e \}$; 
	\item $h_{(n, \ell, i)} \in H_2 - \{ e \}$, for $\ell < k(n)$ and $i = 1, 2$;
	\item $(h_{(n, \ell, i)} : \ell < k(n) \text{ and } i = 1, 2)$ is with no repetitions;
	\item $(h_{(n, 0, 2)})^{-1} \neq h_{(n, k(n)-1, 1)} \neq (h_{(n, 1, 2)})^{-1}$;
	\item for $i \in \{ 1, 2 \}$ and $0 < \ell < k(n)$ we have $h_{(n, 0, 2)} \neq (h_{(n, \ell, 2)})^{-1}$;
	\item for $i \in \{ 1, 2 \}$ and $0 \leq \ell < k(n) - 1$ we have $h_{(n, k(n) - 1, 1)} \neq (h_{(n, \ell, 1)})^{-1}$;
	\item for $i = 1, 2$ we have:
	$$g_{(n, i)} = h_{(n, 0, i)} g_{(n, *)} h_{(n, 1, i)} g^{-1}_{(n, *)} \cdots h_{(n, k(n) - 2, i)} g_{(n, *)} h_{(n, k(n) - 1, i)} g^{-1}_{(n, *)};$$
\end{enumerate}
	\item $\Omega = \{ x_n = g_{(n, 1)} (x_{n+1})^{p(n)} g_{(n, 2)} : n < \omega \}$.
	\end{enumerate}
\end{proposition}

	\begin{proof} Let $(t_n : n < \omega)$ witness the solvability of $\Omega$ in $H$. Notice that:
	\begin{equation}\label{equation_prelim1}
	 \exists n^*_0 \text{ such that } n \geq n^*_0 \text{ implies } t_n \text{ is not } (H_2, H_1)\text{-cyclically reduced}.
	\end{equation}
[Why? Let $n^*_0 = lg(t_0) +1$, and, toward contradiction, assume that $n \geq n^*_0$ and $t_n$ is $(H_2, H_1)$-cyclically reduced. By downward induction on $\ell \leq n$ we can prove that $t_\ell$ is $(H_2, H_1)$-cyclically reduced and $lg(t_\ell) \geq lg(t_n) + n - \ell$. For $\ell = n$, this is clear. For $\ell < n$, by the inductive hypothesis we have that $t_{\ell + 1}$ is $(H_2, H_1)$-cyclically reduced and $lg(t_{\ell + 1}) \geq lg(t_n) + n - (\ell + 1)$. Now, by Proposition~\ref{prop_(ab)_reduced} applied to $(g_1, u, g_2) = (g_{(\ell, 1)}, t_{\ell + 1}, g_{(\ell, 2)})$, we have that $g_{\ell, 1} (t_{\ell + 1})^{p(\ell)} g_{\ell, 2} = t_{\ell}$ is $(H_2, H_1)$-cyclically reduced and $lg(t_{\ell}) > lg(t_{\ell + 1})$, from which it follows that $lg(t_\ell) \geq lg(t_n) + n - \ell$, as wanted. Hence, letting $\ell = 0$ we have that $lg(t_0) \geq lg(t_n) + n \geq n^*_0 > lg(t_0)$, a contradiction.]

\noindent Thus, we have:
	\begin{equation}\label{equation_prelim2}
	\text{ for } n \geq n^*_0 \text{ we have } lg(t_n) > lg(t_{n+1}) \text{ or } lg(t_n) = lg(t_{n+1}) \wedge lg(t_n) \geq 2k(n).
	\end{equation}
[Why? By Proposition \ref{pre_prop_2}(B) applied to $(g_1, u, g_2) = (g_{n, 1}, t_{n+1}, g_{n, 2})$, as case (B)(b) of Proposition \ref{pre_prop_2} is excluded by (\ref{equation_prelim1}).]
\newline Now, by (\ref{equation_prelim2}), we get:
\begin{equation}\label{equation_prelim3}
	(lg(t_n) : n \geq n_*) \text{ is non-increasing}.
\end{equation}
Thus, by (\ref{equation_prelim3}), we get:
\begin{equation}\label{equation_prelim4}
	(lg(t_n) : n \geq n_*) \text{ is eventually constant}.
\end{equation}
Hence, by the second half of (\ref{equation_prelim2}) and (\ref{equation_prelim4}), we contradict assumption $(b)$.
\end{proof}

	We will also need the following results of abelian group theory. We follow \cite{fuchs}.
	
	\begin{definition}\label{def_tor} Let $G$ be an abelian group.
	\begin{enumerate}[(1)]
	\item For $1 \leq n < \omega$, we denote by $Tor_n(G)$ the set of $g \in G$ such that $ng = 0$ (in \cite{fuchs} this is denoted as $G[n]$, cf. pg. 4).	
	\item For $1 \leq n < \omega$, we say that $G$ is $n$-bounded if $Tor_n(G) = G$ (cf. \cite[pg. 25]{fuchs}).
	\item We say that $G$ is bounded if it is $n$-bounded for some $1 \leq n < \omega$ (cf. \cite[pg. 25]{fuchs}).
	\item We say that $G$ is divisible if for every $g \in H$ and $n < \omega$ there exists $h \in G$ such that $nh = g$ (cf. \cite[pg. 98]{fuchs}).
	\item We say that $G$ is reduced if it has no divisible subgroups other than $0$ (cf. \cite[pg. 200]{fuchs}).
	\end{enumerate}
\end{definition}

	\begin{fact}[\cite{fuchs}{[Theorem 23.1]}]\label{pre_fuchs_fact} Let $G$ be a divisible abelian group and $P = \{ p : p \text{ prime} \}$. Then:
$$G \cong \bigoplus_{\alpha < \lambda_{\infty}} \mathbb{Q} \oplus \bigoplus_{p \in P} \bigoplus_{\alpha < \lambda_{p}} \mathbb{Z}^{\infty}_p.$$
\end{fact}

	\begin{fact}[\cite{fuchs}{[Theorem 17.2]}]\label{pre_fuchs_fact2}  Let $G$ be a bounded abelian group. Then $G$ is a direct sum of cyclic groups.
\end{fact}

	\begin{fact}\label{pre_fuchs_fact3} Let $G$ be an abelian group and $1 \leq n < \omega$. Then $Tor_n(G)$ is the direct sum of finite cyclic groups or order divisible by $n$.
\end{fact}

	\begin{proof} This is an immediate consequence of Fact \ref{pre_fuchs_fact2}.
\end{proof}
	
	\begin{definition}\label{fact_bounded_div} Let $G$ be an abelian group and $P = \{ p : p \text{ prime} \}$. 
	\begin{enumerate}[(1)]
	\item For $1 \leq n < \omega$, we say that $G$ is $n$-bounded-divisible when:
	$$G \cong \bigoplus_{\alpha < \lambda_{\infty}} \mathbb{Q} \oplus \bigoplus_{p \in P} \bigoplus_{\alpha < \lambda_{p}} \mathbb{Z}^{\infty}_p \oplus \bigoplus_{p^m \mid n} \bigoplus_{\alpha < \lambda_{p, m}} \mathbb{Z}_{p^m}.$$
	\item We say that $G$ is bounded-divisible if it is $n$-bounded-divisible for some $1 \leq n < \omega$.
	\end{enumerate}	
\end{definition}

\begin{fact}[\cite{fuchs}{[pg. 200]}]\label{max_div_subgr} Let $G$ be an abelian group. Then for some $H\leq G$ (unique up to isomorphism) we have:
 \begin{enumerate}[(1)]
	\item $G$ has a unique maximal divisible subgroup $Div(G)$;
	\item $G = Div(G) \oplus H$;
	\item $H$ is reduced.
\end{enumerate}
\end{fact}

	\begin{fact}\label{fuchs_fact} Let $G$ be an abelian group and $1 \leq n < \omega$. If for every $g \in G$ there exists a divisible $K \leq G$ such that $g \in K + Tor_n(G)$, then $G$ is $n$-bounded-divisible.
\end{fact}

	\begin{proof} This is an immediate consequence of Facts \ref{pre_fuchs_fact}, \ref{pre_fuchs_fact3} and \ref{max_div_subgr}.
\end{proof}

	\begin{fact}\label{last_abelian_fact_second_venue} Let $G$ be a group, $1 \leq n < \omega$ and (for ease of notation) $G' = Cent(G)$. Suppose that both $G/G'$ and $G'/(Div(G') + Tor_n(G'))$ are countable. Then $G = K \oplus M$, with $K$ countable and $M$ bounded-divisible.
\end{fact}

	\begin{proof}
	By Fact \ref{max_div_subgr}, $G' = Div(G') \oplus H$, with $H$ reduced. Furthermore, by assumption, $G'/(Div(G') + Tor_n(G'))$ is countable. So we can find a sequence $(g_i : i < \theta \leq \aleph_0)$ of members of $G'$ such that $G'$ is the union of $(g_i + (Tor_n(G) + Div(G')): i < \theta)$. Thus, since also $G/G'$ is countable, we can find $K \leq G$ such that:
	\begin{enumerate}[(a)]
	\item $K$ is countable;
	\item $G = \bigcup \{ G'h : h \in K \}$;
	\item $K$ includes $\{ g_i : i < \theta \}$.
\end{enumerate}
Now, by Facts \ref{pre_fuchs_fact} and \ref{pre_fuchs_fact3}, $L : = Div(G') + Tor_n(G')$ can be represented as $\bigoplus_{i < \lambda} G_i$ with each $G_i \cong \mathbb{Q}$ or $G_i \cong \mathbb{Z}_{p^{\ell}}$ (with $p^{\ell} \mid n$, for some $1 \leq n < \omega$). Without loss of generality, for some countable $\mathcal{U} \subseteq \lambda$ we have $K \cap L = \bigoplus \{ G_i : i \in \mathcal{U} \}$. Let $M = \bigoplus \{ G_i : i \in \lambda - \mathcal{U} \}$, and notice that:
	\begin{enumerate}[(1)]
	\item $K$ and $M$ commute (since $M \subseteq G'$);
	\item $K + M = G$;
	\item $K \cap M = \{ e\}$.
	\end{enumerate}
Hence, $G = K \oplus M$ and so we are done.
\end{proof}
	
%	So we can find a sequence $(g_i : i < \theta)$ of members of $H$ such that $H$ is the union of $(g_i + Tor_n(H) + Div(G'): i < \theta)$. Let $H_0$ be the subgroup of $G$ generated by $(g_i : i < \theta)$. Now, for every $g \in H$ for some $i < \theta$ we have $g \in g_i + Tor_n(H)$, which means $g - g_i \in Tor_n(H)$, i.e. $H \models n(g - g_i) = 0$; and so $g = g_i + (g - g_i) \in H_0 + Tor_n(H)$. By Fact \ref{pre_fuchs_fact3}, $Tor_n(H) = \bigoplus \{ G_i : i < \lambda \}$, for some $\lambda$ and some cyclic subgroups $G_i$ of $H$ each of order dividing $n$. As $|H_0 \cap Tor_n(H)| \leq |H_0| \leq \theta + \aleph_0 = \aleph_0$, there is a countable subset $\mathcal{U}$ of $\lambda$ such that $H_0 \cap Tor_n(H) \subseteq \bigoplus \{ G_i : i \in \mathcal{U} \}$. Let $H_1 = H_0 + \bigoplus \{ G_i : i \in \mathcal{U} \}$ and $K_1 = \bigoplus \{ G_i : i \in \lambda - \mathcal{U} \}$. Now, it is easy to check that:
%\[ \begin{array}{rcl}
%					H & =   & H_1 + Tor_n(H) \\
%					  & =   & H_1 + \bigoplus \{ G_i : i < \lambda \} \\
%					  & =   & H_0 + \bigoplus \{ G_i : i \in \mathcal{U} \} + K_1 \\
%					  & =   & H_1 + K_1.
%\end{array}	\]
%Furthermore, $H_1 \cap K_1 = \{ 0 \}$. Hence, we have:
%$$G = Div(G) \oplus H = Div(G) \oplus H_1 \oplus K_1 = H_1 \oplus Div(G) \oplus K_1.$$ Now $H_1$ is countable, whereas $G_* = Div(G) \oplus K_1$ is bounded-divisible, as wanted.

	Finally, we will make a crucial use of the following special case of \cite[3.1]{shelah_1}. 
	
\begin{fact}[\cite{shelah_1}]\label{771_fact} Let $G = (G, \mathfrak{d})$ be a Polish group and $\bar{g} = (\bar{g}_n : n < \omega)$, with $\bar{g}_n \in G^{\ell(n)}$ and $\ell(n) < \omega$.
	\begin{enumerate}[(1)]
	\item For every non-decreasing $f \in \omega^\omega$ with $f(n) \geq 1$ and $(\varepsilon_n)_{n < \omega} \in (0, 1)^{\omega}_{\mathbb{R}}$ there is a sequence $(\zeta_n)_{n < \omega}$ (which we call an $f$-continuity sequence for $(G, \mathfrak{d}, \bar{g})$, or simply an $f$-continuity sequence) satisfying the following conditions:
	\begin{enumerate}[(A)]
	\item for every $n < \omega$:
	\begin{enumerate}[(a)]
	\item $\zeta_n \in (0, 1)_{\mathbb{R}}$ and $\zeta_n < \varepsilon_n$;
	\item $\zeta_{n+1} < \zeta_{n}/2$;
	\end{enumerate}
	\end{enumerate}
	\begin{enumerate}[(B)]
	\item for every $n < \omega$, group term $\sigma(x_0, ..., x_{m-1}, \bar{y}_n)$ and $(h_{(\ell, 1)})_{\ell < m}, (h_{(\ell, 2)})_{\ell < m} \in G^m$, the $\mathfrak{d}$-distance from $\sigma(h_{(0, 1)}, ..., h_{(m-1, 1)}, \bar{g}_n)$ to $\sigma(h_{(0, 2)}, ..., h_{(m-1, 2)}, \bar{g}_n)$ is $< \zeta_n$, when: 
	\begin{enumerate}[(a)]
	\item $m \leq n+1$;
	\item $\sigma(x_0, ..., x_{m-1}, \bar{y}_n)$ has length $\leq f(n)+1$; 
	\item $h_{(\ell, 1)}, h_{(\ell, 2)} \in Ball(e; \zeta_{n+1})$;
	\item $G \models \sigma(e, ..., e, \bar{g}_n) = e$.
	\end{enumerate}
	\end{enumerate}
	\item The set of equations $\Gamma = \{ x_n = d_{(n, 1)} (x_{n+1})^{k(n)} d_{(n, 2)} : n < \omega \}$ is solvable in $G$ when for every $n < \omega$:
	\begin{enumerate}[(a)]
	\item $f \in \omega^\omega$ is non-decreasing and $f(n) \geq 1$;
	\item $1 \leq k(n) < f(n)$;
	\item $(\zeta_n)_{n < \omega}$ is an $f$-continuity sequence;
	\item $\mathfrak{d}(d_{(n, \ell)}, e) < \zeta_{n+1}$, for $\ell = 1, 2$.
	\end{enumerate}
	\end{enumerate}
\end{fact}

	\begin{convention}\label{convention} If we apply Fact \ref{771_fact}(1) without mentioning $\bar{g}$ it means that we apply Fact \ref{771_fact}(1) for $\bar{g}_n = \emptyset$, for every $n < \omega$.
\end{convention}

	We shall use the following observation freely throughout the paper.

	\begin{observation}\label{observation_prelim} Suppose that $(G, \mathfrak{d})$ is Polish, $A \subseteq G^k$ is uncountable, $1 \leq k < \omega$ and $\zeta > 0$. Then for some $(g_{1, \ell} : \ell < k) = \bar{g}_1 \neq \bar{g}_2 = (g_{2, \ell} : \ell < k) \in A$ we have $\mathfrak{d}((g_{1, \ell})^{-1}g_{2, \ell}, e) < \zeta$, for every $\ell < k$. 
\end{observation}

	\begin{proof} We give a proof for $k = 1$, the general case is similar. First of all, notice that we can find $g_1 \in A$ such that $g_1$ is an accumulation point of $A$, because otherwise we contradict the separability of $(G, \mathfrak{d})$. Furthermore, the function $(x, y) \mapsto x^{-1}y$ is continuous and so for every $(x_1, y_1) \in G^2$ and $\zeta > 0$ there is $\delta > 0$ such that, for every $(x_2, y_2) \in G^2$, if $\mathfrak{d}(x_1, x_2), \mathfrak{d}(y_1, y_2) < \delta$, then $\mathfrak{d}((x_1)^{-1}y_1, (x_2)^{-1}y_2) < \zeta$. Let now $g_2 \in Ball(g_1; \delta) \cap A - \{g_1\}$, then $\mathfrak{d}((g_1)^{-1}g_2, (g_1)^{-1}g_1) = \mathfrak{d}((g_1)^{-1}g_2, e) < \zeta$.
\end{proof}

\section{First Venue}\label{first_venue}

	In this section we prove Theorem \ref{th_first_venue}. We will prove a series of lemmas from which the theorem follows.

	\begin{lemma}\label{crucial_lemma} Let $\Gamma$ be such that either of the following cases happens:
\begin{enumerate}[(i)]
	\item in $\Gamma$ there are $\{ a_i : i < \omega_1 \}$ and $\{ b_i : i < \omega_1 \}$ such that if $i < j < \omega_1$, then $a_i \neq a_j$, $b_i \neq b_j$, $|\{ a_i, a_j, b_i, b_j \}| = 4$ and $a_i$ is not adjacent to $b_i$;
	\item in $\Gamma$ there are $a_*$ and $\{ b_i : i < \omega_1 \}$ such that if $i < j < \omega_1$, then $|\{ a_*, b_i, b_j \}| = 3$ and $a_*$ is not adjacent to $b_i$.
\end{enumerate} 
Then $G(\Gamma, G_a)$ does not admit a Polish group topology.
\end{lemma}

	\begin{proof} Suppose that $G = G(\Gamma, G_a) = (G, \mathfrak{d})$ is Polish. 
\newline \underline{\em Case 1}. There are $\{ (a_i, b_i) : i < \omega_1 \}$ as in (i) above.
\newline Let $(\zeta_n)_{n < \omega} \in (0, 1)_{\mathbb{R}}^\omega$ be as in Fact \ref{771_fact} for $f \in \omega^{\omega}$ e.g. constantly $30$ (recall Convention~\ref{convention}). Using Observation \ref{observation_prelim}, by induction on $n < \omega$, choose $(i(n), j(n))$, $(g_{i(n)}, g_{j(n)})$ and $(h_{i(n)}, h_{j(n)})$ such that:
	\begin{enumerate}[(a)]
	\item if $m < n$, then $j_m < i_n$;
	\item $i_n < j_n < \omega_1$;
	\item $g_{i(n)} \in G_{a_{i(n)}} - \{ e \}$ and $g_{j(n)} \in G_{a_{j(n)}} - \{ e \}$;
	\item $h_{i(n)} \in G_{b_{i(n)}} - \{ e \}$ and $h_{j(n)} \in G_{b_{j(n)}} - \{ e \}$; 
	\item $\mathfrak{d}((g_{i(n)})^{-1} g_{j(n)}, e), \mathfrak{d}((h_{i(n)})^{-1} h_{j(n)}, e) < \zeta_{n+4}$.
\end{enumerate}
%For every $n < \omega$, choose $(a_{i(n)}, a_{j(n)})$ and $(b_{i(n)}, b_{j(n)})$ such that ... . 
Consider now the following set of equations:
	$$ \Omega = \{ x_n = (x_{n+1})^{2}(t_n)^{-1} : n < \omega \},$$
where $t_n = ((g_{i(n)})^{-1} g_{j(n)} (h_{i(n)})^{-1} h_{j(n)})^3$. By (e) above and Fact \ref{771_fact}(1)(B) we have $\mathfrak{d}(t_n, e) < \zeta_{n+1}$, and so by Fact \ref{771_fact}(2) the set $\Omega$ is solvable in $G$. Let $(d'_n)_{n < \omega}$ witness this. Now $sp(d'_0)$ is finite, and so we can find $0 < n < \omega$ such that $sp(d'_0) \cap \{ a_{i(n)}, a_{j(n)}, b_{i(n)}, b_{j(n)}\} = \emptyset$. Let now $A = \{ a_{i(n)}, a_{j(n)}, b_{i(n)}, b_{j(n)}\}$, $\mathbf{p} = \mathbf{p}_A$ the corresponding homomorphism from Fact \ref{retract_fact} and let $\mathbf{p}(d'_m) = d_m$. Then we have:
	\begin{enumerate}[(A)]
	\item $d_0 = e$;
	\item $m < n \Rightarrow d_m = (d_{m+1})^{2}\mathbf{p}(t_m) = (d_{m+1})^{2}$.
\end{enumerate}
Thus, $(d_n)^{2^n} = e$. Hence, by Proposition \ref{prop_1}(1), we have that $csp(d_n)$ is a complete graph (and so $|csp(d_n))| \leq 2$). Furthermore, we have:
$$(d_{n+1})^2 = d_nt_n,$$
Hence, we reach a contradiction with Proposition \ref{prop_1}(2).
\newline \underline{\em Case 2}. There is $a_*$ and $\{ b_i : i < \omega_1 \}$ as in (ii) above. 
\newline Let $k(n)$ and $p(n)$ be as in Proposition \ref{prop_2}, $g_* \in G_{a_*} - \{ e \}$ and let $(\zeta_n)_{n < \omega} \in (0, 1)_{\mathbb{R}}^\omega$ be as in Fact \ref{771_fact} for $f \in \omega^{\omega}$ such that $f(n) = p(n) + 4k(n) + 4$ and $\bar{g}_n = (g_*)$ (and so in particular $\ell(n) = 1$).
Using Observation \ref{observation_prelim}, by induction on $n < \omega$, choose $(i(n), j(n)) = (i_n, j_n)$ and $(h_{i(n)}, h_{j(n)})$ such that:

	\begin{enumerate}[(a)]
	\item if $m < n$, then $j_m < i_n$;
	\item $i_n < j_n < \omega_1$;
	\item $h_{i(n)} \in G_{b_{i(n)}} - \{ e \}$ and $h_{j(n)} \in G_{b_{j(n)}} - \{ e \}$;
	\item $\mathfrak{d}((h_{i(n)})^{-1} h_{j(n)}, e) < \zeta_{n+2k(n)+2}$.  
\end{enumerate}
Let $g_{(n, *)} = g_*$, $h_n =  (h_{i(n)})^{-1}h_{j(n)}$, $h_{(n, \ell, 1)} = h_{n! + 2\ell}$ and $h_{(n, \ell, 2)} = h_{n! + (2\ell + 1)}$, for $\ell < k(n)$. Let then $g_{(n, i)}$ and $\Omega$ be as in Proposition \ref{prop_2}. By (e) above and Fact \ref{771_fact}(1)(B) we have $\mathfrak{d}(g_{(n, i)}, e) < \zeta_{n+1}$. Thus, by Fact \ref{771_fact}(2) the set $\Omega$ is solvable in $G$.
Let now $A = \{ a_{_*} \} \cup \{ b_{i(n)}, b_{j(n)} : n < \omega \}$ and $\mathbf{p} = \mathbf{p}_A$ be the corresponding homomorphism from Fact \ref{retract_fact}. Then projecting onto $\mathbf{p}(G) = G(\Gamma \restriction A, G_a)$ and using Proposition \ref{prop_2} we get a contradiction, since, for every $n < \omega$, $a_*$ is adjacent to neither $b_{i(n)}$ nor $b_{j(n)}$, and so $G(\Gamma \restriction A, G_a) = G_{a_*} \ast G(\Gamma \restriction A - \{a_* \}, G_a)$.
\end{proof}

	As a corollary of the previous lemma we get:

	\begin{corollary}\label{complete_outside_countable} Let $G = G(\Gamma, G_a)$. If $G$ admits a Polish group topology, then there exists a countable $A_1 \subseteq \Gamma$ such that for every $a \in \Gamma$ and $a \neq b \in \Gamma - A_1$, $a$ is adjacent to $b$.
\end{corollary}

	\begin{lemma}\label{not_abelian_is_countable} If the set:
	$$A_2 = \{a \in \Gamma : G_a \text{ is not abelian}\}$$
is uncountable, then $G(\Gamma, G_a)$ does not admit a Polish group topology.
\end{lemma}

	\begin{proof} Suppose that $G = G(\Gamma, G_a) = (G, \mathfrak{d})$ is Polish, and let $A_1 \subseteq \Gamma$ be as in Corollary \ref{complete_outside_countable} (recall that $A_1$ is countable). By induction on $n$, choose $(a_n, g_n, t_n)$, $(b_n, d_n, z_n)$, $(h_n, h_{< n})$ and $(\zeta^{\ell}_n : \ell = 1, ..., 4)$ such that: 
	\begin{enumerate}[(a)]
	\item $a_n \neq b_n \in A_2 - (A_1 \cup \{ a_{\ell}, b_{\ell} : \ell < n\})$;
	\item $g_n, t_n \in G_{a_n}$ and they do not commute;
	\item $d_n, z_n \in G_{b_n}$ and they do not commute;
	\item $\mathfrak{d}((g_n)^{-1} d_n), e), \mathfrak{d}((t_n)^{-1} z_n, e) < \zeta^4_n$;
	\item $h_n = (g_n)^{-1}d_n$ and $h_{< n} = h_0 \cdots h_{n-1}$;
	\item $\zeta^{\ell}_n \in (0, 1)_{\mathbb{R}}$, $\frac{1}{4} \zeta^{\ell}_n \geq \zeta^{\ell + 1}_n$ and $\frac{1}{4} \zeta^{4}_n \geq \zeta^{1}_{n+1}$;
	\item if $n = m+1$ and $g \in Ball(h_{<n}; \zeta^2_n)$, then $g$ and $(t_m)^{-1}z_m$ do not commute;
	\item if $n = m+1$, and $g \in Ball(e; \zeta^3_n)$, then $\mathfrak{d}(h_{<n}g, h_{<n}) \leq \zeta^2_m$.
\end{enumerate}
[How? For $n = 0$, let $\zeta^{\ell}_n = \frac{1}{4^{\ell + 1}}$, and choose $(a_0, g_0, t_0)$, $(b_0, d_0, z_0)$, $(h_0, h_{< 0})$ as needed (where we let $h_{< 0} = e$). So assume $n = m+1$, and let $\zeta^1_n = \frac{1}{4}\zeta^4_m$. 
Now, $(g_m, d_m)$ are well-defined, and so $h_{< n} = h_{< m}h_m$ is well-defined.
Furthermore, $h_{< n}$ does not commute with $(t_m)^{-1}z_m$, i.e. $h_{< n}(t_m)^{-1}z_m(h_{< n})^{-1}(z_m)^{-1}t_m \neq e$. Thus, there is $\zeta^2_n \in (0, \frac{1}{4}\zeta^1_n)_{\mathbb{R}}$ such that:
$$g \in Ball(h_{< n}, \zeta^2_n) \; \Rightarrow \; g(t_m)^{-1}z_mg^{-1}(z_m)^{-1}t_m \neq e.$$
Also, let $\zeta^3_n \in (0, \frac{1}{4}\zeta^2_n)_{\mathbb{R}}$ be as in Fact \ref{771_fact}(1)(B) with $(\zeta^2_n, \zeta^3_n, 4)$ here standing for $(\zeta_n, \zeta_{n+1}, f(n))$ there. Similarly, choose $\zeta^4_n \in (0, \frac{1}{4}\zeta^3_n)_{\mathbb{R}}$. Finally, we show how to choose $(a_n, g_n, t_n)$ and $(b_n, d_n, z_n)$. For every $a \in A_2 - (A_1 \cup \{ a_{\ell}, b_{\ell}: \ell < n \})$ we have that $G_a$ is not abelian, and so we can find $g^a_n, p^a_n \in G_a$ which do not commute. Since $A_2$ is uncountable whereas $A_1 \cup \{ a_{\ell}, b_{\ell}: \ell < n \}$ is countable and $(G, \mathfrak{d})$ is separable, we can find uncountable $A'_n \subseteq A_2 - (A_1 \cup \{ a_{\ell}, b_{\ell}: \ell < n \})$ and $g^*_n$ such that $\{ g^a_n : a \in A'_n \} \subseteq Ball(g^*_n, \zeta^4_n/2)$. Similarly, we can find uncountable $A''_n \subseteq A'_n$ and $p^*_n$ such that $\{ p^a_n : a \in A'_n \} \subseteq Ball(p^*_n, \zeta^4_n/2)$. Chose $a_n \neq b_n \in A''_n$ and let:
$$g_n = g^a_n, \; t_n = p^a_n, \; g^b_n = d_n \text{ and } p^b_n = z_n.$$
Then $(a_n, g_n, t_n)$, $(b_n, d_n, z_n)$, $(h_n, h_{< n})$ and $(\zeta^{\ell}_n : \ell = 1, ..., 4)$ are as wanted.]
\newline Then we have: 
\begin{enumerate}[(A)] 
	\item $(h_{< n} : n < \omega)$ is Cauchy, let its limit be $h_{\infty}$;
	\item $\mathfrak{d}(h_{\infty}, h_{< n+1}) < \zeta^1_n$;
	\item $h_{\infty}$ and $(t_n)^{-1}z_n$ do not commute.
\end{enumerate}
[Why? By clause (d) above, for each $n$ we have $\mathfrak{d}((g_{n+1})^{-1}d_{n+1}, e) < \zeta^4_{n+1} < \zeta^3_{n+1}$, and so by clause (h) we have $\mathfrak{d}(h_{<n+1} , h_{<n}) \leq \zeta^2_n$. Furthermore, $\zeta^2_{n+1} < \zeta^1_{n+1} \leq \frac{1}{4}\zeta^4_n < \frac{1}{4}\zeta^2_n$. Thus, clearly the sequence $(h_{< n} : n < \omega)$ is Cauchy. Moreover, we have: 
$$\mathfrak{d}(h_{\infty},h_{< n+1}) \leq \sum\{ \zeta^2_k : k \geq n\} \leq 2\zeta^2_n < \zeta^1_n,$$
so clause (B) is satisfied.
Finally, clause (C) follows by (B) and clause (g) above.]
\newline Let $n < \omega$ be such that $\{ a_n, b_n \} \cap sp(h_{\infty}) = \emptyset$. Then $h_{\infty}$ and $(t_n)^{-1}z_n$ commute (cf. the choice of $A_1$), contradicting (C).
\end{proof}

	\begin{lemma}\label{div_mod_tor} Let $G = G(\Gamma, G_a)$ and $A_1, A_2 \subseteq \Gamma$ be as in Corollary \ref{complete_outside_countable} and Lemma \ref{not_abelian_is_countable}.  For $n < \omega$, $a \in \Gamma - (A_1 \cup A_2)$ and $g \in G_a$ we write $\varphi_n(g, G_a)$ to mean that for no divisible $K \leq G_a$ we have $g \in K + Tor_n(G_a)$ (cf. Definition \ref{def_tor}).
If for every $n < \omega$ the set:
	$$A_3(n) = \{a \in \Gamma - (A_1 \cup A_2) : \exists g \in G_a \text{ such that } \varphi_n(g, G_a) \}$$
is uncountable, then $G$ does not admit a Polish group topology.
\end{lemma}

	\begin{proof} Suppose that $G = G(\Gamma, G_a) = (G, \mathfrak{d})$ is Polish, and let $(\zeta_n)_{n < \omega} \in (0, 1)_{\mathbb{R}}^\omega$ be as in Fact \ref{771_fact} for $f \in \omega^{\omega}$ such that $f(n) = n + 4$. By induction on $n < \omega$, choose $(a(n), b(n))$ and $(g_{a(n)}, g_{b(n)})$ such that:
	\begin{enumerate}[(a)]
	\item $a(n) \neq b(n) \in \Gamma - (A_1 \cup A_2 \cup \{ a({\ell}), b({\ell}) : \ell < n\})$;
	\item $g_{a(n)} \in G_{a(n)} - \{ e \}$ and $g_{b(n)} \in G_{b(n)} - \{ e \}$;
	\item for no divisible $K \leq G_{a(n)}$ we have $g_{a(n)} \in K + Tor_{n!}(G_{a(n)})$;
	\item $\mathfrak{d}((g_{b(n)})^{-1} g_{a(n)}, e) < \zeta_{n+1}$.
\end{enumerate}
	Consider now the following set of equations: 
	$$ \Omega = \{ x_n = (x_{n+1})^{n+1} h_n: n < \omega \},$$
where $h_n = (g_{b(n)})^{-1} g_{a(n)}$. By (d) above we have $\mathfrak{d}(h_n, e) < \zeta_{n+1}$, and so by Fact \ref{771_fact}(2) the set $\Omega$ is solvable in $G$. Let $(d'_n)_{n < \omega}$ witness this. Let then $0 < n < \omega$ be such that $sp(d'_0) \cap \{ a(n), b(n) \} = \emptyset$. Let now $A = \{ a_n \}$, $\mathbf{p} = \mathbf{p}_A$ the corresponding homomorphism from Fact \ref{retract_fact} and let $\mathbf{p}(d'_n) = d_n$. Then we have (in additive notation):
\begin{enumerate}[(i)]
	\item $d_0 = e$;
	\item $m \neq n \Rightarrow d_m = (m + 1) d_{m+1} + \mathbf{p}(h_m) = (m + 1) d_{m+1}$;
	\item $d_n = (n + 1) d_{n+1} + \mathbf{p}(h_n) = (n + 1) d_{n+1} + g_{a(n)}$.
\end{enumerate}
Thus, by (ii) for $m < n$ we have $n! d_n = 0$, i.e.:
\begin{equation}\label{equation_first_venue}
d_n \in Tor_{n!}(G_{a(n)}).
\end{equation}
Furthermore, by (ii) for $m > n$ the subgroup $K$ of $G_{a(n)}$ generated by $\{ d_{n+1}, d_{n+2}, ... \}$ is divisible. Hence, by (iii) and (\ref{equation_first_venue}) we have:
$$g_{a(n)}  = - (n+1)d_{n+1} + d_n \in K + Tor_{n!}(G_{a(n)}),$$
which contradicts the choice of $g_{a(n)}$.
\end{proof}

	\begin{definition}\label{tor_def_number} Let $G = G(\Gamma, G_a)$. We define (recalling the notation of Lemma \ref{div_mod_tor}): 
	\begin{enumerate}[(1)]
	\item $n(G) = min\{ m \geq 2 : \text{ for all but $\leq \aleph_0$  many } a \in \Gamma, \forall g \in G_a (\neg \varphi_m(g, G_a)) \};$
	\item $A_3 = \{ a \in \Gamma : G_a \text{ is abelian and } \exists g \in G_a (\varphi_{n(G)}(g, G_a))\}$.
\end{enumerate}
\end{definition}

	\begin{corollary}\label{bound_div} Let $G = G(\Gamma, G_a)$, and suppose that $G$ admits a Polish group topology. Then:
	\begin{enumerate}[(1)]
	\item the natural number $n = n(G)$ from Definition \ref{tor_def_number}(1) is well-defined;
	\item the set $A_3$ from Definition \ref{tor_def_number}(2) is countable;
	\item the set $A_4 = \{a \in \Gamma: G_a \text{ is abelian and not $n$-bounded-divisible}\}$ is countable.
\end{enumerate}
\end{corollary}

	\begin{proof} This follows from Lemma \ref{div_mod_tor} and Fact \ref{fuchs_fact}.
\end{proof}

	\begin{lemma}\label{lemma_finitely_many} Suppose that $G = G(\Gamma, G_a)$ admits a Polish group topology and let $A_1, ..., A_4$ be as Corollary \ref{complete_outside_countable}, Lemma \ref{not_abelian_is_countable}, Definition \ref{tor_def_number} and Corollary \ref{bound_div}. Then there exists a countable $A \subseteq \Gamma$ and $n < \omega$ such that:
	\begin{enumerate}[(a)]
	\item  $A_1 \cup \cdots \cup A_4 \subseteq A$;
	\item  if $a \in \Gamma - A$, then $G_a$ is $n$-bounded-divisible.
\end{enumerate}	 
\end{lemma}

	\begin{proof} This is because of Corollaries \ref{complete_outside_countable} and \ref{bound_div}, and Lemmas \ref{not_abelian_is_countable} and \ref{div_mod_tor}.
\end{proof}

	\begin{lemma}\label{elimination_Z_infty_p} Let $G = G' \oplus G''$, with $G'' = \bigoplus_{\alpha < \lambda} G_\alpha$, $\lambda > \aleph_0$ and $G_\alpha \cong \mathbb{Z}^{\infty}_p$ (for $\mathbb{Z}^{\infty}_p$ cf. Notation \ref{def_bounded}). Then $G$ does not admit a Polish group topology.
\end{lemma}

	\begin{proof} Suppose that $G = (G, \mathfrak{d})$ is Polish, and that $G = G' \oplus G''$ is as in the assumptions of the lemma. Let $(\zeta_n)_{n < \omega} \in (0, 1)_{\mathbb{R}}^\omega$ be as in Fact \ref{771_fact} for $f \in \omega^{\omega}$ such that $f(n) = p^{k(n)} + 1$, $k(n) > n$ and $2nk(n) < k(n+1)$.
For every $n < \omega$, choose $(\alpha(n),  \beta(n))$ and $(g_n, h_n)$ such that:
	\begin{enumerate}[(a)]
	\item $\alpha(n) < \beta(n) < \lambda$ and $\alpha(n), \beta(n) \notin \{ \alpha(\ell), \beta(\ell) : \ell < n \}$;
	\item $g_n \in G_{\alpha(n)}$ and $h_n \in G_{\beta(n)}$;
	\item $g_n$ and $h_n$ have order $p^{nk(n)}$ (so $g_0 = e = h_0$);
	\item $\mathfrak{d}((g_n)^{-1} h_n, e) < \zeta_{n+1}$.
\end{enumerate} 
Consider now the following set of equations:
	$$ \Omega = \{ x_n = (x_{n+1})^{p^{k(n)}} t_n : n < \omega \},$$
where $t_n = (g_n)^{-1} h_n$. By (d) above we have $\mathfrak{d}(t_n, e) < \zeta_{n+1}$, and so by Fact \ref{771_fact}(2) the set $\Omega$ is solvable in $G$. Let $(d'_n)_{n < \omega}$ witness this. Let then $\mathbf{p}$ be the natural projection from $G$ onto $G_* = \bigoplus_{n < \omega} G_{\beta_n}$ (cf. Fact \ref{retract_fact}), and set $d_n = \mathbf{p}(d'_n)$. Hence, for every $n < \omega$, we have (in additive notation):
$$G_* \models d_n = p^{k(n)} d_{n+1} + h_n,$$
and so:
\begin{equation}\label{equation_elimination}
G_* \models d_0 = h_0 + p^{k(0)}h_1 + p^{k(0) + k(1)} h_2 + \cdots + p^{\sum_{\ell < n} k(\ell)} h_n + p^{\sum_{\ell \leq n} k(\ell)} h_{n + 1}.
\end{equation}
Thus, multiplying both sides of (\ref{equation_elimination}) by $p^{n k(n)}$, we get:
\begin{equation}\label{sec_equation_elimination}
G_* \models p^{n k(n)}d_0 = p^{\sum_{\ell \leq n} k(\ell)} p^{n k(n)} h_{n + 1},
\end{equation}
since, for $\ell \leq n$, $h(\ell)$ has order $p^{\ell k(\ell)}$ and $\ell k(\ell) \leq n k(n)$, and so we have $p^{nk(n)} h_{\ell} = 0$. Notice now that that the right side of (\ref{sec_equation_elimination}) is $\neq 0$, since $p^{\sum_{\ell \leq n} k(\ell)} p^{n k(n)}$ divides $p^{nk(n)} p^{nk(n)} = p^{2nk(n)}$, $2nk(n) < k(n+1) < (n+1)k(n+1)$ and the order of $h_{n+1}$ is $p^{(n+1) k(n + 1)}$. Hence, also the left side of (\ref{sec_equation_elimination}) is $\neq 0$, but this is contradictory, since $G_*$ is an abelian $p$-group and $k(n) > n$, for every $n < \omega$.
\end{proof}

	The next lemma is stronger than what needed for the proof of Theorem \ref{th_first_venue}, we need this formulation for the proof of Theorem \ref{th_second_venue}.

	\begin{lemma}\label{card_invariants} Suppose that $G$ admits a Polish group topology, $G = G_1 \oplus G_2$, $G_1$ is countable and $G_2 = \bigoplus \{ G^*_{s, \lambda_{s}} : s \in S_* \}$ (cf. Notation \ref{def_bounded}). Then for every $s \in S_*$ we have that $\lambda_s$ is either $\leq \aleph_0$ or $2^{\aleph_0}$. 
\end{lemma}

	\begin{proof} Let $G = (G, \mathfrak{d})$ be Polish and $G = G_1 \oplus G_2$ be as in the assumptions of the lemma. Then $G_2 \cong \bigoplus \{ G_t : t \in I \}$, where for each $t \in I$ we have $G_t \cong G^*_s$ for some $s \in S_*$. For $s \in S_*$, let $I_s = \{ t \in I : G_t \cong G^*_s\}$. So $(I_s : s \in S_*)$ is a partition of $I$. We want to show that for each $s \in S_*$ we have that $|I_s| \leq \aleph_0$ or $|I_s| = 2^{\aleph_0}$. Since $S_*$ is countable,  $|I_s| \leq |G|$ and $(G, \mathfrak{d})$ is Polish, it suffices to show that $|I_s| > \aleph_0$ implies $|I_s| = 2^{\aleph_0}$. Notice that the case $s = (p, n)$ is actually taken care of by Lemma 18 and Observation 19  of \cite{paolini&shelah_cyclic}, but for completeness of exposition we give a direct proof also in the case $s = (p, n)$. 
\newline \noindent For $s \in S_*$ and $t \in I_s$, let $g_t \in G_t - \{ e \}$ be such that $g_t$ satisfies no further demands in the case $s =\infty$, and $g_t$ generates $G_t$ in the case $s = (p, n)$. Now, fix $s \in S_*$ and, using Observation \ref{observation_prelim}, by induction on $n < \omega$, choose:
$$(a(n), b(n), g_{a(n)}, g_{b(n)}, (h_{\mathcal{U}} : \mathcal{U} \subseteq n), h_n, \zeta^1_n, \zeta^2_n),$$ such that:
\begin{enumerate}[(a)]
	\item $h_{\mathcal{U}} = \prod_{\ell \in \mathcal{U}} h_{\ell}$;
	\item $0 < \zeta^1_n < \zeta^2_n < 1$;
	\item if $\mathcal{U} \subseteq n$ and $g \in Ball(e; \zeta^2_n)$, then $\mathfrak{d}(h_{\mathcal{U}}g, h_{\mathcal{U}}) < \zeta^1_n$;
	\item $a(n) \neq b(n) \in I_s - \{ a(\ell), b(\ell) : \ell < n \}$;
	\item $h_n =  (g_{a(n)})^{-1} g_{b(n)}$;
	\item $\mathfrak{d}(h_n, e) < \zeta^2_n$;
	\item $\zeta^2_{n+1} < \frac{1}{2} \zeta^1_{n}$.
\end{enumerate}
Then for $\mathcal{U} \subseteq \omega$ we have that $(h_{\mathcal{U} \cap n} : n < \omega)$ is a Cauchy sequence. Let $h_{\mathcal{U}}$ be its limit.
\newline \underline{\em Case 1}. $s = \infty$.
\newline Let: 
$$E_{\infty} = \{ (\mathcal{U}_1, \mathcal{U}_2) : \mathcal{U}_1, \mathcal{U}_2 \subseteq \omega \text{ and } \exists n \geq 2 \text{ and }  \exists g \in G_1 ((h_{\mathcal{U}_1}(h_{\mathcal{U}_2})^{-1})^n g^{-1} = e )\}.$$
Notice that:
	\begin{enumerate}[(i)]
	\item $E_{\infty}$ is an equivalence relation on $\mathcal{P}(\omega)$;
	\item $E_{\infty}$ is analytic (actually even Borel, recalling $G_1$ is countable);
	\item $\mathcal{U}_1, \mathcal{U}_2 \subseteq \omega$ and $\mathcal{U}_2 - \mathcal{U}_1 = \{ m \}$, then $\neg(\mathcal{U}_1 E_{\infty} \mathcal{U}_2)$.
\end{enumerate}
Hence, by \cite[Lemma 13]{sh_for_CH}, we get $(\mathcal{U}_{\alpha} : \alpha < 2^{\aleph_0})$ such that the $h_{\mathcal{U}_{\alpha}}$'s are pairwise non $E_{\infty}$-equivalent. Notice now that $\bigoplus \{ G_t : t \not\in I_{\infty} \}$ is torsion, while the $h_{\mathcal{U}_{\alpha}}$'s have infinite order.
Furthermore, by the choice of $E_{\infty}
$ we have that $\alpha < \beta < 2^{\aleph_0}$ implies that for every $n \geq 2$ we have $((h_{\mathcal{U}_{\alpha}}(h_{\mathcal{U}_{\beta}})^{-1})^n \not\in G_1$. It follows that:
$$G/(\bigoplus \{ G_t : t \not\in I_{\infty} \} \oplus G_1)$$
has cardinality $2^{\aleph_0}$, and so $|I_{\infty}| = 2^{\aleph_0}$, as wanted.
\newline \underline{\em Case 2}. $s = (p, n)$.
\newline Let:
$$E_{(p, n)} = \{ (\mathcal{U}_1, \mathcal{U}_2) : \mathcal{U}_1, \mathcal{U}_2 \subseteq \omega \text{ and } (h_{\mathcal{U}_1}(h_{\mathcal{U}_2})^{-1})^{p^{n - 1}} \in G_1 + pG_2 )\}.$$
Notice that:
	\begin{enumerate}[(i)]
	\item $E_{(p, n)}$ is an equivalence relation on $\mathcal{P}(\omega)$;
	\item $E_{(p, n)}$ is analytic (actually even Borel, recalling $G_1$ is countable);
	\item $\mathcal{U}_1, \mathcal{U}_2 \subseteq \omega$ and $\mathcal{U}_2 - \mathcal{U}_1 = \{ m \}$, then $\neg(\mathcal{U}_1 E_{(p, n)} \mathcal{U}_2)$.
\end{enumerate}
Hence, by \cite[Lemma 13]{sh_for_CH}, we get $(\mathcal{U}_{\alpha} : \alpha < 2^{\aleph_0})$ such that the $h_{\mathcal{U}_{\alpha}}$'s are pairwise non $E_{(p, n)}$-equivalent. Notice now that:
\begin{equation}\label{equation_card_invariant-1}
(h_{\mathcal{U}_a})^{p^{n}} = e \; \text{ and } \; (h_{\mathcal{U}_a})^{p^{n-1}} \neq e.
\end{equation}
%and so we have:
%$$h_{\mathcal{U}_a} \in \bigoplus \{ G_t : t \in I_{(p, m)}, m \geq n \} \oplus G_1.$$ 
Furthermore, by the choice of $E_{(p, n)}$ we have that:
\begin{equation}\label{equation_card_invariant}
\alpha < \beta < 2^{\aleph_0} \text{ implies } (h_{\mathcal{U}_{\alpha}}(h_{\mathcal{U}_{\beta}})^{-1})^{p^{n - 1}} \not\in G_1 + pG_2.
\end{equation}
Let $\mathbf{p}$ be the projection of $G$ onto $G_2$ (cf. Fact \ref{retract_fact}), and for $\alpha < 2^{\aleph_0}$ let $\mathbf{p}(h_{\mathcal{U}_{\alpha}}) = h'_{\alpha}$. Thus, by (\ref{equation_card_invariant}), we get:
\begin{equation}\label{equation_card_invariant2}
\alpha < \beta < 2^{\aleph_0} \text{ implies } (h'_{\alpha}(h'_{\beta})^{-1})^{p^{n - 1}} \neq e.
\end{equation}
Thus, from (\ref{equation_card_invariant-1}) and (\ref{equation_card_invariant2}) it follows that:
$$Tor_{p^n}(G_2)/(Tor_{p^{n-1}}(G_2) + pG_2)$$
has cardinality $2^{\aleph_0}$, and so $|I_{(p, n)}| = 2^{\aleph_0}$, as wanted.
\end{proof}

	\begin{proof}[Proof of Theorem \ref{th_first_venue}] This follows directly from Lemma \ref{lemma_finitely_many} (cf. the definitions of $A_1, ..., A_4$ there), Lemma \ref{elimination_Z_infty_p} (recalling Definition \ref{fact_bounded_div}) and Lemma \ref{card_invariants}.
\end{proof}

\section{Second Venue}\label{second_venue}

	In this section we prove Theorem \ref{th_second_venue}. As in the previous section, we will prove a series of lemmas from which the theorem follows. %If the arguments needed to prove our lemmas are analogous to the ones used in the previous section we might omit details.
	
	\begin{lemma}\label{lemma_unc_factors} If $G = G(\Gamma, G_a)$, $a \neq b \in \Gamma$, $\{ a, b \} \notin E_{\Gamma}$ and $G_b$ is uncountable, then $G$ does not admit a Polish group topology.
\end{lemma}

	\begin{proof} Suppose that $G = G(\Gamma, G_a) = (G, \mathfrak{d})$ is Polish, and let $a \neq b \in \Gamma$ be as in the assumptions of the lemma. Let $k(n)$ and $p(n)$ be as in Proposition \ref{prop_2}, $g_* \in G_{a} - \{ e \}$ and let $(\zeta_n)_{n < \omega} \in (0, 1)_{\mathbb{R}}^\omega$ be as in Fact \ref{771_fact} for $f \in \omega^{\omega}$ such that $f(n) = p(n) + 4k(n) + 4$ and $\bar{g}_n = (g_*)$ (and so in particular $\ell(n) = 1$). Using Observation \ref{observation_prelim}, by induction on $n < \omega$, choose $h_n$ such that:
	\begin{enumerate}[(a)]
	\item $e \neq h_n \in G_b - \{ h_{\ell} : \ell < n \}$; 
	\item $\mathfrak{d}(h_{n}, e) < \zeta_{n+2k(n)+2}$.
\end{enumerate}
Let $g_{(n, *)} = g_*$, $h_{(n, \ell, 1)} = h_{n! + 2\ell}$ and $h_{(n, \ell, 2)} = h_{n! + (2\ell + 1)}$, for $\ell < k(n)$. Let then $g_{(n, i)}$ and $\Omega$ be as in Proposition \ref{prop_2}. By (b) above and Fact \ref{771_fact}(1)(B) we have $\mathfrak{d}(g_{(n, i)}, e) < \zeta_{n+1}$, and so by Fact \ref{771_fact}(2) the set $\Omega$ is solvable in $G$.
Let now $A = \{ a, b \}$ and $\mathbf{p} = \mathbf{p}_A$ be the corresponding homomorphism from Fact \ref{retract_fact}. Then projecting onto $\mathbf{p}(G) = G(\Gamma \restriction A, G_a)$ and using Proposition \ref{prop_2} we get a contradiction, since $a$ is not adjacent to $b$, and so $G(\Gamma \restriction A, G_a) = G_{a} \ast G_{b}$.
\end{proof}

	\begin{definition}\label{def_A0} For $\Gamma$ a graph, let:
	$$A_0 = A_0(\Gamma) = \{ a \in \Gamma : \text{for some } b \in \Gamma - \{a \} \text{ we have } \{a, b \} \not\in E_{\Gamma} \}.$$
\end{definition}

	\begin{lemma}\label{finitely_many_uncount_factors} If the set:
	$$A_5 = \{a \in \Gamma - A_0: G_a \text{ is not abelian and } [G_a : Cent(G_a)] \text{ is  uncountable}\}$$
is infinite, then $G(\Gamma, G_a)$ does not admit a Polish group topology.
\end{lemma}

	\begin{proof} Suppose that $G = G(\Gamma, G_a) = (G, \mathfrak{d})$ is Polish and that the set $A_5$ in the statement of the lemma is infinite. Let then $\{ a(n) : n < \omega\}$ be an enumeration of $A_5$ without repetitions. First of all, notice that for every $a \in \Gamma$ such that $[G_a : Cent(G_a)]$ is  uncountable we have:
\begin{equation}\label{cent_equation}
	\text{for every } \varepsilon \in (0, 1)_{\mathbb{R}} \text{ we have } Ball(e; \varepsilon) \cap G_a \not\subseteq Cent(G_a).
\end{equation}
Now, by induction on $n < \omega$, choose $(g_{n, 1}, g_{n, 2}, (h_{\mathcal{U}} : \mathcal{U} \subseteq n), \zeta^2_n, \zeta^1_n)$ such that:
\begin{enumerate}[(a)] 
	\item $h_{\mathcal{U}} = \prod_{\ell \in \mathcal{U}} h_{\ell}$;
	\item $\zeta^1_n < \zeta^2_n \in (0, 1)_{\mathbb{R}}$, and for $n = m+1$ we have $\zeta^2_n < \frac{\zeta^1_m}{4}$;
	\item if $h \in Ball(e; \zeta^2_{n+1}) \cap G_{a(n)}$ and $\mathcal{U} \subseteq n$, then $\mathfrak{d}(h_{\mathcal{U}}h, h_{\mathcal{U}}) < \zeta^1_n$;
	\item $g_{n, 1} \in (Ball(e; \zeta^2_n) \cap G_{a(n)}) - Cent(G_{a(n)})$, $g_{n, 2} \in G_{a(n)}$ and $g_{n, 1}$ and $g_{n, 2}$ do not commute;
	\item if $h \in Ball(g_{n, 1}; \zeta^1_n) \cap G_{a(n)}$, then $h \in Ball(e; \zeta^2_n) \cap G_{a(n)}$, and $h$ and $g_{n, 2}$ do not commute;
	\item $h_n = g_{n, 1}$.
\end{enumerate}
[How? First choose $\zeta^2_n$ satisfying clauses (b) and (c). Then, using (\ref{cent_equation}), choose $g_{n, 1} = h_n $ as in clause (d). Finally, choose $\zeta^1_n \in (0, \zeta^2_n)_{\mathbb{R}}$ as in clause (e).]
\newline For $n < \omega$, let $h_{< n} = h_0 \cdots h_{n-1}$. Then $(h_{< n} : n < \omega)$ is Cauchy, let its limit be $h_{\infty}$. 
Notice now that because of Lemma \ref{lemma_unc_factors} without loss of generality we can assume that $n < m < \omega$ implies $\{ a(n), a(m) \} \in E_{\Gamma}$, and also that if $b \in \Gamma - \{ a(n) \}$ then $a(n) E_{\Gamma} b$. For $n < m$, let $h_{n , m} = h_n \cdots h_m$ and $h_{n, \infty} = lim (h_{n, m} : n < m < \omega)$.
Let now $n < \omega$ be such that $sp(h_{\infty}) \cap \{ a(n) \} = \emptyset$. Then we have:
\begin{enumerate}[(a')]
	\item $g_{n, 2}$ and $h_{n}$ do not commute;
	\item $g_{n, 2}$ commutes with $h_0, ..., h_{n-1}$ and with $h_{n + 1, \infty}$;
	\item $h_{\infty} = h_0 \cdots h_{n - 1} h_n h_{n + 1, \infty}$;
	\item $h_{\infty}$ and $g_{n, 2}$ do not commute.
\end{enumerate}
[Why? Clause (a') is by the inductive choices (a)-(f). Clause (b') is because for $\ell < n$ we have $a({\ell}) E_{\Gamma} a(n)$. Clause (c') is easy.  Clause (d') is an immediate consequence of (a'), (b') and (c').]
\newline Thus, by (d') we get a contradiction, since $sp(h_{\infty}) \cap \{ a(n) \} = \emptyset$, $g_{n, 2} \in G_{a(n)}$ and $b \in \Gamma - \{ a(n) \}$ implies $a(n) E_{\Gamma} b$.
\end{proof}

	\begin{lemma}\label{finiteA6} For $G$ a group, we write $\psi(G)$ to mean that $[G : Cent(G)]$ is countable, and (for ease of notation) we let $G' = Cent(G)$. If for every $n < \omega$ the set (recalling Fact \ref{max_div_subgr} and Definition \ref{def_tor}):
	$$A_6(n) = \{a \in \Gamma - A_0: \psi(G_a) \text{ and } G'_a/(Div(G'_a) + Tor_n(G'_a)) \text{ is uncountable}\}$$
is infinite, then $G(\Gamma, G_a)$ does not admit a Polish group topology.
\end{lemma}

	\begin{proof} Suppose that $G = G(\Gamma, G_a) = (G, \mathfrak{d})$ is Polish, and let $A^*_6 = \bigcup_{n < \omega} A_6(n)$.
	Notice now that:
	\begin{enumerate}[(a)]
	\item\label{a} $a \in A^*_6$ implies $a \notin A_0(\Gamma)$ (cf. Definition \ref{def_A0});
	\item\label{b} $(Cent(G), \mathfrak{d} \restriction Cent(G))$ is a Polish group;
	\item $Cent(G) \subseteq G(\Gamma \restriction B, G_a)$, where $B = \Gamma - A_0(\Gamma)$; 
	\item $G(\Gamma \restriction B, G_a) = \bigoplus_{a \in B} G_a$;
	\item $Cent(\bigoplus_{a \in B} G_a) = \bigoplus_{a \in B} Cent(G_a) = G(\Gamma \restriction B, Cent(G_a))$.
\end{enumerate}
[Why? (\ref{a}) is because of Lemma \ref{lemma_unc_factors}. (\ref{b}) is because the commutator function is continuous and a closed subgroup of a Polish group is Polish. The rest is clear.]
\newline Hence it suffices to prove the lemma for the abelian case, i.e. assume that $\Gamma$ is complete and all the factors groups $G_a$ are abelian.
%	Suppose that $G = G(\Gamma, G_a) = (G, \mathfrak{d})$ is Polish. 
%Let $A^*_6 = \bigcup_{n < \omega} A_6(n)$ and $G^* = G_{A^*_6} = G(\Gamma \restriction A^*_6, G_a)$. Now, $a \in A^*_6$ implies $a \notin A_0$, and so $G^*$ is a direct summand of $G$. Furthermore, by Lemma \ref{not_abelian_is_countable} and Corollary \ref{bound_div}, for some $n_*$ we have that $A^*_6(n_*)$ is countable. For each $a \in A^*_6(n_*)$, we know that $[G : Cent(G)]$ is countable and so there is a countable subgroup $H_a$ of $G_a$ such that $G_a = H_a + Cent(G_a)$ (as for every $g_1 \in G_a$ there exists $g_2 \in H_a$ such that $g_1 (g_2)^{-1} \in Cent(G_a)$). Furthermore, clearly the group $H_a$ is a normal subgroup of $G$ because $a \in A_0(\Gamma)$. Hence $H_* : = \sum_{a \in A^*_6(n_*)} H_a$ is normal in $G$. Furthermore, $G_1 = G + K$, hence $H_*$ is normal in $G_1$. Now, as $H_*$ is countable and $(G, \mathfrak{d})$ is Polish clearly $Cent(H_*; G_1)$ is a closed subgroup of $G_1$
Let then $(\zeta_n)_{n < \omega} \in (0, 1)_{\mathbb{R}}^\omega$ be as in Fact \ref{771_fact} for $f \in \omega^{\omega}$ such that $f(n) = n + 4$. Toward contradiction, assume that for every $n < \omega$ the set $A_6(n)$ is infinite. Then we can choose $a(n) \in \Gamma - \{ a({\ell}) : \ell < n \}$ such that $a(n) \in A_6(n!)$, by induction on $n$. So we can find $g_{n, \alpha} \in G_{a(n)} - \{ e \}$, for $\alpha < \omega_1$, such that:
\begin{equation}\label{equation_sec_venue}
(g_{n, \alpha} + (Div(G_{a(n)}) + Tor_{n!}(G_{a(n)})) : \alpha < \omega_1) \; \text{ are pairwise distinct}.
\end{equation}
By induction on $n < \omega$, choose $\alpha(n) < \beta(n) < \omega_1$ such that $\mathfrak{d}((g_{n, \alpha(n)})^{-1} g_{n, \beta(n)}, e) < \zeta_{n+1}$. Then $h_n = (g_{n, \beta(n)})^{-1} g_{n, \alpha(n)} \in G_{a(n)}$ satisfies:
	\begin{enumerate}[(a)]
	\item $\mathfrak{d}(h_n, e) < \zeta_{n + 1}$;
	\item $h_n \notin Div(G_{a(n)}) + Tor_{n!}(G_{a(n)})$.
	\end{enumerate}
[Why? Clause (a) is clear. Clause (b) is by (\ref{equation_sec_venue}).]
\newline Consider now the following set of equations:
	$$ \Omega = \{ x_n = (x_{n+1})^{n+1} h_n: n < \omega \}.$$
By (a) above and Fact \ref{771_fact}(2) the set $\Omega$ is solvable in $G$. Let $(d'_n)_{n < \omega}$ witness this. Let then $0 < n < \omega$ be such that $sp(d'_0) \cap \{ a(n) \} = \emptyset$. Let now $A = \{ a(n) \}$, $\mathbf{p} = \mathbf{p}_A$ the corresponding homomorphism from Fact \ref{retract_fact} and let $\mathbf{p}(d'_n) = d_n$. Then we have (in additive notation):
\begin{enumerate}[(i)]
	\item $d_0 = e$;
	\item $m \neq n \Rightarrow d_m = (m + 1) d_{m+1} + \mathbf{p}(h_m) = (m + 1) d_{m+1}$;
	\item $d_n = (n + 1) d_{n+1} + \mathbf{p}(h_n) = (n + 1) d_{n+1} + h_n$.
\end{enumerate}
Thus, by (ii) for $m < n$ we have $n! d_n = 0$, i.e.:
\begin{equation}\label{equation_end_sec_venue}
d_n \in Tor_{n!}(G_{a(n)}).
\end{equation}
Furthermore, by (ii) for $m > n$ the subgroup $K$ of $G_{a(n)}$ generated by $\{ d_{n+1}, d_{n+2}, ... \}$ is divisible. Hence, by (iii) and (\ref{equation_end_sec_venue}) we have:
$$h(n)  = - (n+1)d_{n+1} + d_n \in K + Tor_{n!}(G_{a(n)}),$$
which contradicts (b) above.
\end{proof}

	We now have all the ingredients for proving Theorem \ref{th_second_venue}.

	\begin{proof}[Proof of Theorem \ref{th_second_venue}] Suppose that $G = G(\Gamma, G_a)$ admits a Polish group topology, and let $n$ be minimal such that $A_6(n)$ is finite (cf. Lemma \ref{finiteA6}). We define (notice that $A_6$ below is in fact $A_6(n)$):
	\begin{enumerate}[(i)]
	\item $A_0 = \{ a \in \Gamma : \text{for some } b \in \Gamma - \{a \} \text{ we have } \{a, b \} \not\in E_{\Gamma} \}$;
	\item $A_5 = \{a \in \Gamma : G_a \text{ is not abelian and } [G_a : Cent(G_a)] \text{ is  uncountable}\}$;
	\item $A_6 = \{a \in \Gamma : \psi(G_a) \text{ and } G'_a/(Div(G'_a) + Tor_n(G'_a)) \text{ is uncountable}\}$;
	\item $A_7 = \{ a \in \Gamma : a \not\in A_0 \cup A_5 \cup A_6 \text{ and $G_a$ is not abelian} \}$;
	\item $A_8 = \{ a \in \Gamma : a \not\in A_0 \cup A_5 \cup A_6 \text{ and $G_a$ is abelian and not bounded-divisible} \}$;
	\item $A_9 = \{ a \in \Gamma : a \not\in A_0 \cup A_5 \cup A_6 \text{ and $G_a$ is abelian and bounded-divisible} \}$.
\end{enumerate}
We claim that $\bar{A} = (A_0, A_5, A_6, A_7, A_8, A_9)$ is as wanted, i.e. we verify clauses (\ref{0})-(\ref{i}) of the statement of the theorem.  Clauses (\ref{0}), (\ref{1}) and (\ref{i}) are clear. Clause (c) is by Lemmas \ref{finitely_many_uncount_factors} and \ref{finiteA6}. Clause (\ref{c}) for $A_0$ is by Lemma \ref{crucial_lemma}, for $A_7$ is by Lemma \ref{not_abelian_is_countable} and for $A_8$ is by Corollary \ref{bound_div}.
Clause (\ref{d}) is by Lemma \ref{lemma_unc_factors}. Clause (\ref{e}) is by  Fact \ref{last_abelian_fact_second_venue}. Clause (\ref{f}) is by Definition \ref{fact_bounded_div}. Clause (\ref{g}) is by Lemma \ref{elimination_Z_infty_p}. Clause (\ref{n_bound}) is by Lemma \ref{lemma_finitely_many}, modulo renaming the factor groups $G_a$ (if necessary).
\noindent Finally, we want to show that assuming CH and letting $A = A_0 \cup A_7 \cup  A_8 \cup  A_9$ we have that $G_A = G(\Gamma \restriction A, G_a)$ admits a non-Archimedean Polish group topology. By clauses (\ref{0})-(\ref{i}) of the statement of the theorem we have:
	$$G_A \cong H \oplus \bigoplus_{\alpha < \lambda_{\infty}} \mathbb{Q} \oplus \bigoplus_{p^n \mid n_*} \bigoplus_{\alpha < \lambda_{(p, n)}} \mathbb{Z}_{p^n},$$
for some countable $H$ and $\lambda_{\infty}, \lambda_{(p, n)} \in \{ 0, 2^{\aleph_0} \}$. Since finite sums of groups admitting a non-Archimedean Polish group topology admit a non-Archimedean Polish group topology, it suffices to show that $H_1 = \bigoplus_{\alpha < 2^{\aleph_0}} \mathbb{Q} \cong \mathbb{Q}^{\omega}$ and $H_2 = \bigoplus_{\alpha < 2^{\aleph_0}} \mathbb{Z}_{p^n} \cong \mathbb{Z}_{p^n}^{\omega}$ admit one such topology. Let $K$ be either $\mathbb{Q}$ or $\mathbb{Z}_{p^n}$, and let $A$ be a countable first-order structure such that $Aut(A) = K$. Let $B$ be the disjoint union of $\aleph_0$ copies of $A$, then $K^{\omega} \cong Aut(B)$, and so we are done.
\end{proof}

	\begin{proof}[Proof of Theorem \ref{new_th_second_venue}] The fact that (1)(a) (resp. (2)(a)) implies (1)(b) (resp. (2)(b)) is clear. Concerning the other implications, argue as in the proof of Theorem \ref{th_second_venue}.
\end{proof}
%	\begin{proof}[Proof of Theorem \ref{new_th_second_venue}] Item (2) is clear. Concerning item (1), the fact that (1)(b) implies (1)(a) is also clear. Thus, suppose (1)(a) holds, and let $\bar{A}$ be as in the proof of Theorem \ref{th_second_venue}, then we can find finite $B \subseteq \Gamma$ such that $A_5 \cup A_6 \subseteq B$ and if $B \subseteq B_1 \subseteq B \cup A_7 \cup A_8 \cup A_9$ and $s \in S_*$, then $\sum \{ \lambda_{a, s} : a \in (A_7 \cup A_8 \cup A_9) - B_1 \} = \sum \{ \lambda_{a, s} : a \in (A_7 \cup A_8 \cup A_9) - B \}$, and so we are done.
%\end{proof}

\section{Third Venue}\label{third_venue}

	In this section we prove Corollaries \ref{cor_third_venue}, \ref{cor_bounded_divis} and \ref{corollary_free_prod}.
	
	\begin{proof}[Proof of Corollary \ref{cor_third_venue}] By Theorem \ref{th_first_venue} and Lemma \ref{card_invariants} the necessity of the conditions is clear. Concerning the sufficiency, argue as in the proof of Theorem \ref{th_second_venue}.
\end{proof}

	\begin{proof}[Proof of Corollary \ref{cor_bounded_divis}] This is an immediate consequence of Corollary \ref{cor_third_venue}.
\end{proof}

	\begin{proof}[Proof of Corollary \ref{corollary_free_prod}] This is a consequence of Corollary \ref{complete_outside_countable} and Lemma \ref{lemma_unc_factors}.
\end{proof}

\end{document}